\documentclass[EJP]{ejpecp}

\usepackage{appendix} % to create an appendix with proper labels
\usepackage[numbers]{natbib} % to use \citet
\usepackage{enumitem} % to control enumerate's labelling and reference items
\usepackage{graphicx} % to define the \bigtimes command below

\SHORTTITLE{Gaussian product inequality conjecture for Wishart random matrices}

\TITLE{On the Gaussian product inequality conjecture for disjoint principal minors of Wishart random matrices}

\AUTHORS{
Christian~Genest\footnote{Department of Mathematics and Statistics, McGill University, Montr\'eal (Qu\'ebec) Canada H3A 0B9.
\BEMAIL{christian.genest@mcgill.ca}} % \BEMAIL has a built-in line break
\and
Fr\'ed\'eric~Ouimet\footnote{Department of Mathematics and Statistics, McGill University, Montr\'eal (Qu\'ebec) Canada H3A 0B9.
\BEMAIL{frederic.ouimet2@mcgill.ca}} % \BEMAIL has a built-in line break
\and
Donald~Richards\footnote{Department of Statistics, Pennsylvania State University, University Park, PA 16802, USA.
\BEMAIL{richards@stat.psu.edu}} % \BEMAIL has a built-in line break
}

 % Separate items with ;

\AMSSUBJ{60E15} % Separate items with ;
\AMSSUBJSECONDARY{26A48; 44A10; 62E15; 62H10} % Optional, separate items with ;

\SUBMITTED{March 12, 2024} % Edit.
\ACCEPTED{***} % Edit.

\VOLUME{0}
\YEAR{2024}
\PAPERNUM{0}
\DOI{10.1214/YY-TN}

\ABSTRACT{This paper extends various results related to the Gaussian product inequality (GPI) conjecture to the setting of disjoint principal minors of Wishart random matrices. This includes product-type inequalities for matrix-variate analogs of completely monotone functions and Bernstein functions of Wishart disjoint principal minors, respectively. In particular, the product-type inequalities apply to inverse determinant powers. Quantitative versions of the inequalities are also obtained when there is a mix of positive and negative exponents. Furthermore, an extended form of the GPI is shown to hold for the eigenvalues of Wishart random matrices by virtue of their law being multivariate totally positive of order~2 ($\mathrm{MTP}_2$). A new, unexplored avenue of research is presented to study the GPI from the point of view of elliptical distributions.}

\newcommand{\N}{\mathbb{N}}
\newcommand{\R}{\mathbb{R}}
\newcommand{\PP}{\mathsf{P}} % Russian style, do not change
\newcommand{\EE}{\mathsf{E}} % Russian style, do not change
\newcommand{\bb}[1]{\boldsymbol{#1}}
\newcommand{\rd}{\mathrm{d}}
\newcommand{\etr}{\mathrm{etr}}
\newcommand{\bigtimes}{\scalebox{1.1}{$\times$}}
\newcommand{\midoplus}{\scalebox{1.2}{~$\oplus$~}}

\begin{document}

\section{Introduction}\label{sec:intro}

The Gaussian product inequality (GPI) is a long-standing conjecture credited to \citet{MR1636556}, who stated it in 1998.  The GPI was initially framed in the context of the real linear polarization constant problem in functional analysis. \citet{MR2385646} reformulated it 10~years later as a moment inequality for Gaussian random vectors. In the latter form, the conjecture asserts that for any centered Gaussian random vector $\bb{Z} = (Z_1, \ldots, Z_d)$ of dimension $d\in \N = \{1,2,\ldots\}$ and every integer $m\in \N_0 = \{0,1,\ldots\}$, one~has
\begin{equation}\label{eq:1.1}
\EE \left(\prod_{i=1}^d Z_i^{2 m}\right) \geq \prod_{i=1}^d \EE \left( Z_i^{2 m} \right).
\end{equation}

The inequality~\eqref{eq:1.1} is known to imply the real linear polarization constant conjecture; see \citet{MR3425898}. It is also closely related to the $U$-conjecture, which states that if two polynomials in the components of $\bb{Z}$ are independent, then there exists an orthogonal transformation such that the transformed polynomials are dependent on disjoint subsets of the components of $\bb{Z}$; see, e.g., \cite{MR0346969,MR3425898}, and references therein.

\citet{MR2886380} proposed a strong version of the conjecture which claims that for every dimension $d\in \N$ and every vector $\bb{\alpha} = (\alpha_1, \ldots, \alpha_d)\in [0,\infty)^d$, one has
\begin{equation}\label{eq:1.2}
\mathrm{GPI}_d(\bb{\alpha}) : \quad \EE \left( \prod_{i=1}^d |Z_i|^{2 \alpha_i} \right) \geq \prod_{i=1}^d \EE \left( |Z_i|^{2 \alpha_i} \right).
\end{equation}
This inequality is homogeneous in the sense that each random variable $Z_i$ can be multiplied at will by a positive constant $k_i$ without altering the conclusion. For this reason, one assumes, without loss of generality, that the covariance matrix of $\bb{Z}$, denoted by $\Sigma = (\sigma_{ij})_{1\leq i,j \leq d}$, is a correlation matrix, i.e., $\sigma_{ii} = 1$ for every integer $i \in \{1,\ldots, d\}$.

Here is a point-by-point list, ordered chronologically, of what is known about~\eqref{eq:1.2}:
\begin{enumerate}[label=(\alph*),start=1]
\item
For all $\bb{\alpha}\in [0,\infty)^2$, $\mathrm{GPI}_2(\bb{\alpha})$ holds as a consequence of $(|Z_1|, |Z_2|)$ being $\mathrm{MTP}_2$; see \citet{MR628759}. \label{item:a}

\item
For all $d\in \N$, $\mathrm{GPI}_d(1,\ldots, 1)$ holds, as proved by \citet{MR2385646}, who exploited Wick's formula in conjunction with properties of permanents and Hafnians. \label{item:b}

\item
For all $m,n\in \N_0$, $\mathrm{GPI}_3(m, m, n)$ holds, as proved by \citet{MR4052574}, based on a careful analysis of Gaussian hypergeometric functions. \label{item:c}

\item
For all $d\in \N$ and $\bb{n}\in \N_0^d$, $\mathrm{GPI}_d(\bb{n})$ holds whenever the correlation matrix $\Sigma$ is completely positive, as proved in \cite{MR4466643}, based on a combinatorial methodology closely tied to the log-convexity of the gamma function. \label{item:d}

\item
For all $d\in \N$ and $\bb{n}\in \N_0^d$, $\mathrm{GPI}_d(\bb{n})$ holds whenever all the entries of the correlation matrix $\Sigma$ are nonnegative, as proved by \citet{MR4445681}, using the Isserlis--Wick type formula of \citet{MR3324071}; see \citet{arXiv:1705.00163} for a more detailed proof.

This result was generalized by \citet{MR4554766} to the case in which the squared Gaussian random variables $Z_1^2, \ldots, Z_d^2$ are replaced by the components of a multivariate Gamma random vector in the sense of \citet{MR44790}. In the version given in \citep{MR4554766}, the correlation matrix $\Sigma$ is allowed to have nonnegative entries up to a conjugation transformation $S\mapsto S \Sigma S$ by a signature matrix $S$. As highlighted in \cite{MR4538422}, this last result can be generalized even further to random vectors of traces of disjoint diagonal blocks of Wishart random matrices. The multivariate Gamma distribution corresponds to the special case in which the diagonal blocks have size $1\times 1$. \label{item:e}

\item
For all $(n_1, n_2, n_3)\in \{1\} \times \{2, 3\} \times \N_0$, $\mathrm{GPI}_3(n_1,n_2,n_3)$ holds, as validated by \citet{MR4445681}, who employed a brute-force combinatorial approach. \label{item:f}

\item
For all $(n_1, n_2, n_3)\in \N_0 \times \{3\} \times \{2\}$, $\mathrm{GPI}_3(n_1,n_2,n_3)$ holds, as stated in Theorem~4.1 of \citet{MR4760098}, taking advantage of a sums-of-squares methodology combined with in-depth symbolic calculations using \texttt{Macaulay2} and \texttt{Mathematica}. \label{item:g}

\item
For all $(n_1, n_2, n_3, n_4)\in \N_0 \times \{1\} \times \{1\} \times \{1\}$, $\mathrm{GPI}_4(n_1,n_2,n_3,n_4)$ holds, as stated in Theorem~4.2 of \citet{MR4760098}, using the same methodology as for~\ref{item:g}. The case $(\alpha_1, \alpha_2, \alpha_3, \alpha_4, \alpha_5)\in [1/20,\infty) \times \{1\} \times \{1\} \times \{1\} \times \{1\}$ is treated in Theorem~5.1 of the same paper. \label{item:h}

\item
For all $d\in \N$ and $\bb{n}\in \N_0^d$, $\mathrm{GPI}_d(\bb{n})$ is investigated in \cite{Ouimet_2023_multinomial_GPI} when $\Sigma$ is a multinomial correlation matrix, and it is shown that the GPI is equivalent to a combinatorial inequality only involving finitely many terms. \label{item:i}

\item
For all $(n_1, n_2, n_3)\in \{1\} \times \N_0 \times \N_0$, $\mathrm{GPI}_3(n_1,n_2,n_3)$ holds, as stated in Theorem~1.1 of \citet{MR4593134}, where a stronger moment ratio inequality is proved. \label{item:j}

\item
For all $(n_1, n_2, n_3)\in \N_0^3$, $\mathrm{GPI}_3(n_1, n_2, n_3)$ holds, as proved by \citet{MR4661091}, using induction and an elegant Gaussian integration-by-parts argument. In particular, their result supersedes~\ref{item:c},~\ref{item:f},~\ref{item:g},~and~\ref{item:j} above. \label{item:k}

\item
For all $(\alpha_1, \alpha_2, \alpha_3)\in \N_0 \times [0,\infty) \times [0,\infty)$, $\mathrm{GPI}_3(\alpha_1, \alpha_2, \alpha_3)$ holds, as shown by \citet{MR4798604} who extended the induction approach of \citet{MR4661091} in \ref{item:k}. \label{item:l}

\item
For all $(\alpha_1, \alpha_2, \alpha_3)\in \N_0/2 \times \N_0/2 \times \N_0/2$, $\mathrm{GPI}_3(\alpha_1, \alpha_2, \alpha_3)$ holds, as shown by \citet{MR4794515} who combined the results in \ref{item:k} and \ref{item:l} together with induction and an elaborate brute-force proof of the base case $\mathrm{GPI}_3(1/2, 1/2, 1/2)$. \label{item:m}
\end{enumerate}

One notable complementary direction of research that has emerged in the literature is the case in~\eqref{eq:1.2} for which all exponents are assumed to be nonpositive and the squared Gaussian random variables $Z_1^2, \ldots, Z_d^2$ are replaced by the components of a more general random vector $\bb{X} = (X_1, \ldots, X_d)$ following a given distribution supported on $(0,\infty)^d$. It is also assumed that the right-hand side in~\eqref{eq:1.2} can be written more generally as the multiplication of product moments for any partition of factors from the full product. Specifically, for an appropriate range of $\nu_1, \ldots, \nu_d\, \in [0, \infty)$, and for any integer $k\in \{2, \ldots, d \}$, one considers
\begin{equation}\label{eq:1.3}
\EE \left(\prod_{i=1}^d X_i^{-\nu_i} \right) \geq \EE \left(\prod_{i=1}^{k-1} X_i^{-\nu_i}\right) \EE \left(\prod_{i=k}^d X_i^{-\nu_i}\right),
\end{equation}
provided that the expectations are finite. Here is a list of what is known about \eqref{eq:1.3}:
\begin{enumerate}[label=(\alph*),start=14]
\item
\citet{MR3278931} has shown that~\eqref{eq:1.3} holds for all $\nu_1, \ldots, \nu_d \in [0,\alpha/2)$ when $\bb{X}$ is a multivariate Gamma random vector in the sense of \citet{MR44790} with $\alpha$ degrees of freedom. This result was recovered by \citet{MR4554766} through the application of the multivariate Gamma extension of the Gaussian correlation inequality, due to \citet{MR3289621}, which shows that the random vector $(X_1^{-\nu_1},\ldots,X_d^{-\nu_d})$ is strongly positive upper orthant dependent. In particular, the special case in which the $X_i$'s are the squared components $Z_1^2, \ldots, Z_d^2$ of a Gaussian random vector $\bb{Z}$ holds for every $\nu_1, \ldots, \nu_d \in [0,1/2)$. \label{item:n}

\item
In \cite{MR4538422}, \ref{item:n} was generalized in two ways by showing that~\eqref{eq:1.3} holds when $\bb{X}$ is a random vector of traces of disjoint diagonal blocks of a Wishart random matrix, and the negative power functions $x_i\mapsto x_i^{-\nu_i}$ are replaced by completely monotone functions $\phi_i : (0,\infty) \to [0,\infty)$. In their work, the multivariate Gamma distribution is the special case in which diagonal blocks are $1\times 1$ in size. \label{item:o}
\end{enumerate}

The first objective of the present paper is to generalize statement~\ref{item:n} and complement statement~\ref{item:o} by showing that \eqref{eq:1.3} holds when $\bb{X}$ is a random vector of disjoint principal minors of a Wishart random matrix. In fact, it will be shown more generally that the inequality holds if the inverse powers of disjoint principal minors are replaced by matrix-variate completely monotone functions taking the disjoint diagonal blocks of a Wishart random matrix as arguments; see Theorem~\ref{thm:generalization.Theorem.3.2.Wei} and Corollary~\ref{cor:GPI.negative.powers.determinants.Wishart} for further details. These results rest on the crucial property that the disjoint diagonal blocks of Wishart random matrices are smaller or equal to their independent counterparts in the matrix-variate Laplace transform order; see Proposition~\ref{prop:Wishart.diag.block.Lt.order}.

A complementary direction of research that has emerged in the literature comprises the so-called `opposite GPIs', an expression coined by \citet{MR4471184}. In this setting, the exponents in~\eqref{eq:1.2} are assumed to be a mix of nonpositive and nonnegative numbers and the problem is to determine whether the inequality holds in a specific direction and, if so, under what conditions. \citet{MR4471184} established, by employing a moment formula from \citet{MR0045347}, that the converse of inequality~\eqref{eq:1.2} holds for $d = 2$ and every $(\alpha_1,\alpha_2)\in (-1/2,0] \times [0,\infty)$, thereby completing the exploration of the GPI conjecture in the bivariate case. Tighter bounds were subsequently investigated by \citet{MR4530374} in this context.

Very recently, the $d$-dimensional case was tackled by \citet{MR4666255}. In addition to specialized results in dimensions three and four, these authors showed that for any centered Gaussian random vector $\bb{Z} = (Z_1, \ldots, Z_d)$, any dimension $d \in \N$ and any $\nu_1 \in [0,1/2)$, one has
\begin{equation}\label{eq:1.4}
\EE\left(|Z_1|^{-2\nu_1} \prod_{i=2}^d Z_i^2\right) \geq \left\{\prod_{i=2}^d (1 - \sigma_{1i}^2)\right\} \EE \left( |Z_1|^{-2\nu_1} \right) \prod_{i=2}^d \EE \left( Z_i^2 \right),
\end{equation}
and also, for all $\nu_1,\ldots,\nu_{d-1}\in [0,1/2)$ and $\nu_d\in [0,\infty)$,
\begin{equation}\label{eq:1.5}
\EE\left\{\left(\prod_{i=1}^{d-1} |Z_i|^{-2\nu_i}\right) |Z_d|^{2\nu_d}\right\} \leq \EE\left(\prod_{i=1}^{d-1} |Z_i|^{-2\nu_i}\right) \EE \left( |Z_d|^{2\nu_d} \right).
\end{equation}

The second objective of the present paper is to generalize both \eqref{eq:1.4} and \eqref{eq:1.5} to the setting in which the squared components $Z_1^2, \ldots, Z_d^2$ are replaced by disjoint principal minors of Wishart random matrices; see Propositions~\ref{prop:analogue.Theorem.1.1.Zhou.et.al}~and~\ref{prop:analogue.Theorem.1.3.Zhou.et.al}, respectively. In particular, when the diagonal blocks have size $1\times 1$, the present findings validate the results in the multivariate Gamma setting.

As mentioned in item~\ref{item:o} above, the study of inequality~\eqref{eq:1.3} can be extended to the case in which negative power functions are replaced by completely monotone functions. Another similar idea, which was tackled in \cite{MR4538422}, is to investigate what happens if one replaces completely monotone functions with Bernstein functions (also referred to as Laplace exponents by some probabilists). It was shown in \cite{MR4538422} that if $(X_1, X_2)$ is a random pair consisting of the two disjoint traces inside a Wishart random matrix that has been partitioned into a $2\times 2$ block matrix, then, for any given Bernstein functions $f$ and $g$, one has
\[
\EE \{ f(X_1) g(X_2) \} \geq \EE \{ f(X_1) \} \EE \{ g(X_2) \}.
\]

The third and final objective of the present paper is to generalize this result to the setting in which the traces of disjoint diagonal blocks of a Wishart random matrix are replaced by the corresponding disjoint principal minors; see Theorem~\ref{thm:generalization.Theorem.2.GO.2023}.

The paper's outline is as follows. Some definitions and preliminary results are given in Section~\ref{sec:preliminaries}. The main results are stated in Section~\ref{sec:main.results}. All proofs are deferred to Section~\ref{sec:proofs}. In Section~\ref{sec:elliptical.GPI}, a new, unexplored avenue of research is presented to study the GPI from the point of view of elliptical distributions. Some technical lemmas used in the proofs are relegated to~\ref{app:A}. In \ref{app:B}, necessary and sufficient conditions are derived under which the expectations in Corollary~\ref{cor:GPI.negative.powers.determinants.Wishart} and Proposition~\ref{prop:analogue.Theorem.1.3.Zhou.et.al} are finite. In \ref{app:C}, the integrals appearing in the upper bound in Corollary~\ref{cor:GPI.negative.powers.determinants.Wishart} are investigated.

We conclude this introduction with a new conjecture which seems natural, given the results presented here pertaining to the GPI for disjoint principal minors of Wishart random matrices.

\begin{conjecture}\label{conj:1.1}
Let $p \in \N$ be given. For arbitrary degree-of-freedom $\alpha \in (p-1,\infty)$ and scale matrix $\Sigma\in \mathcal{S}_{++}^p$, let $\mathfrak{X} \sim\mathcal{W}_p (\alpha,\Sigma)$ be a $p \times p$ Wishart random matrix in the sense of Definition~\ref{def:Wishart} below. Write $\mathfrak{X}$ in block matrix form $(\mathfrak{X}_{ij})_{1 \leq i, j \leq d}$, where each block $\mathfrak{X}_{ij}$ has size $p_i \times p_j$, so that $p_1 + \cdots + p_d = p$. Then, for every $\nu_1, \ldots, \nu_d \in [0, \infty)$, one has
\begin{equation}\label{eq:1.6}
\EE(|\mathfrak{X}_{11}|^{\nu_1} \cdots |\mathfrak{X}_{dd}|^{\nu_d}) \geq \EE(|\mathfrak{X}_{11}|^{\nu_1}) \cdots \EE(|\mathfrak{X}_{dd}|^{\nu_d}).
\end{equation}
\end{conjecture}

\begin{remark}
Conjecture~\ref{conj:1.1} was proved to hold for $d = 2$ by \citet{arXiv:2409.14512}, where the joint moment $\EE(|\mathfrak{X}_{11}|^{\nu_1} |\mathfrak{X}_{22}|^{\nu_2})$ is shown to be the product of $\EE(|\mathfrak{X}_{11}|^{\nu_2}) \EE(|\mathfrak{X}_{22}|^{\nu_2})$ and the Gaussian hypergeometric function of matrix argument with certain parameters.
\end{remark}

\begin{remark}\label{rem:exect.finite}
One can verify that the expectations in \eqref{eq:1.6} are always finite by applying H\"older's inequality to the left-hand side and using the moment formulas for standalone principal minors in Lemma~\ref{lem:determinant.power.moments}.
\end{remark}

\section{Definitions and preliminary results}\label{sec:preliminaries}

For any positive integer $p\in \N$, denote by $\mathcal{S}^p$, $\mathcal{S}_+^p$, and $\mathcal{S}_{++}^p$ the spaces of real matrices of size $p\times p$ that are respectively symmetric, (symmetric) nonnegative definite, and (symmetric) positive definite.

Throughout the paper, the notation $\mathrm{tr}(\cdot)$ stands for the trace operator for matrices, $\etr(\cdot) = \exp\{\mathrm{tr}(\cdot)\}$ their exponential trace and $|\cdot|$ their determinant. The operation $\scalebox{1.2}{$\oplus$}$ between two matrices $A$ and $B$ stands for their direct sum, i.e., $A \midoplus B \equiv \mathrm{diag}(A,B)$.

When a matrix $M$ is expressed in the form $M = (M_{ij})$ for some block matrices $M_{ij}$ of appropriate size, the notation $M_{a:b,k:\ell} = (M_{ij})_{a\leq i \leq b, k\leq j \leq \ell}$ will often be used to denote submatrices of $M$ formed by a concatenation of a subset of the blocks. If $a = b$ or $k = \ell$, one simply replaces $a\!:\!b$ with $a$ and $k\!:\!\ell$ with $k$, respectively.

For arbitrary positive integer $p \in \N$, recall that the multivariate gamma function $\Gamma_p$ can be equivalently defined, for every $\nu > (p-1)/2$, by
\[
\Gamma_p (\nu) = \int_{\mathcal{S}_{++}^p} \etr(- X) \, |X|^{\nu - (p + 1)/2} \rd X = \pi^{p(p-1)/4} \prod_{i=1}^p \Gamma \{ \nu - (i - 1)/2\},
\]
where $\rd X$ denotes the Lebesgue measure on $\mathcal{S}_{++}^p$ (see, e.g., \citet[p.~61]{MR652932}), and $\Gamma_1(\cdot)$ reduces to the ordinary gamma function $\Gamma(\cdot)$. Further recall from Section~3.2.1 of \citet{MR652932} the following definition of the Wishart distribution, sometimes also referred to as the Wishart or Wishart--Laguerre ensemble in random matrix theory.

\begin{definition}
\label{def:Wishart}
For arbitrary positive integer $p\in \N$, degree-of-freedom $\alpha \in (p-1,\infty)$ and scale matrix $\Sigma\in \mathcal{S}_{++}^p$, the probability density function of the $\mathcal{W}_p(\alpha,\Sigma)$ distribution is defined for every $X\in \mathcal{S}_{++}^p$, relative to the Lebesgue measure $\rd X$ on $\mathcal{S}_{++}^p$, by
\[
f_{\alpha,\Sigma}(X) = \frac{|X|^{\alpha/2 - (p + 1)/2} \etr(-\Sigma^{-1} X / 2)}{2^{p \alpha / 2} |\Sigma|^{\alpha/2} \Gamma_p(\alpha/2)}.
\]
If a random matrix $\mathfrak{X}$ of size $p\times p$ follows this distribution, one writes $\mathfrak{X} \sim\mathcal{W}_p(\alpha,\Sigma)$.
\end{definition}

The definition below of the Laplace transform on $\mathcal{S}_+^p$ is adapted from Definition 7.2.9 of \citet{MR652932}.

\begin{definition}
\label{def:Lt}
Let $\mu$ be a measure on $\mathcal{S}_+^p$. The Laplace transform of $\mu$, if it exists, is a functional defined, for every $T\in \mathcal{S}^p$, by
\[
\mathcal{L}(\mu)(T) = \int_{\mathcal{S}_+^p} \etr(-T X) \, \mu(\rd X).
\]
If $\mathfrak{X}$ is a $p \times p$ nonnegative definite random matrix with induced probability measure $\mu$ on $\mathcal{S}_{+}^p$, its Laplace transform is then written $\phi_{\mathfrak{X}}$ and defined, for every $T\in \mathcal{S}^p$, by
\[
\phi_{\mathfrak{X}}(T) = \mathcal{L}(\mu)(T).
\]
\end{definition}

\begin{remark}
\label{rem:Wishart.Lt}
Let the parameters $p\in \N$, $\alpha\in (p-1,\infty)$ and $\Sigma\in \mathcal{S}_{++}^p$ be given. If one has $\mathfrak{X}\sim\mathcal{W}_p(\alpha,\Sigma)$ as per Definition~\ref{def:Wishart}, then a renormalization of the density function shows that the Laplace transform of $\mathfrak{X}$ is given, for any $T\in \mathcal{S}^p$ such that $T + \Sigma^{-1}/2\in \mathcal{S}_{++}^p$, by
\[
\phi_{\mathfrak{X}}(T) = | I_p + 2 T \Sigma|^{-\alpha/2}.
\]
In particular, this expression is valid whenever $T\in \mathcal{S}_+^p$.
\end{remark}

In what follows, the notion of completely monotonic function recalled below will also play a role. See, e.g., Definition~1.3 of \citet{MR2978140} and \citet{MR1872377} for a review.

\begin{definition}%[Complete monotonicity]
\label{def:CM}
A map $\phi: (0,\infty)\to [0,\infty)$ is called completely monotone if $\phi$ is infinitely differentiable and satisfies, for all $n\in \N$, $(-1)^n \phi^{(n)} \geq 0$.
\end{definition}

Bernstein's well-known theorem (see, e.g., Theorem~4.1.1 of \citet{MR72370}) establishes that Definition~\ref{def:CM} is equivalent to the existence of a measure $\mu$ on $[0,\infty)$ such that $\phi$ is expressible as the Laplace transform $\mathcal{L}(\mu)$. That is, for every $t\in (0,\infty)$, one has
\begin{equation}\label{eq:2.1}
\phi(t) = \mathcal{L}(\mu)(t) = \int_{[0,\infty)} e^{-t x} \mu(\rd x).
\end{equation}

Hence, the matrix-variate extension of the notion of complete monotonicity below is natural. The equivalence with the corresponding differential version of the definition, which appears, e.g., as Definition~2.1 of \citet{MR3286037}, is due to \citet{MR253009}.

\begin{definition}%[Matrix-variate complete monotonicity -- integral version]
\label{def:MCM}
A map $\phi : \mathcal{S}_{++}^p \to [0,\infty)$ is called matrix-variate completely monotone if there exists a measure $\mu$ on $\mathcal{S}_+^p$ such that $\phi$ is expressible as the Laplace transform $\mathcal{L}(\mu)$ as per Definition~\ref{def:Lt}, i.e., for all $T\in \mathcal{S}_{++}^p$, one has
\[
\phi(T) = \mathcal{L}(\mu)(T) = \int_{\mathcal{S}_+^p} \etr(-T X) \, \mu(\rd X).
\]
\end{definition}

Next, recall the definition of a real univariate Bernstein function (sometimes referred to as a completely monotone mapping or Laplace exponent by some probabilists); see, e.g., Definition~4.1.1 and Theorem~4.1.2 of \citet{MR72370}.

\begin{definition}%[Bernstein function]
\label{def:B}
A map $g: (0,\infty)\to [0,\infty)$ is called a Bernstein function if $g$ is infinitely differentiable and satisfies $(-1)^{n-1} g^{(n)} \geq 0 $ for every integer $n \in \N$.
\end{definition}

This notion is extended below to the matrix-variate setting following the characterization of Bernstein functions on unions of affinely-placed orthants given in Theorem~4.2.4 of \citet{MR72370}.

\begin{definition}%[Matrix-variate Bernstein function -- integral version]
\label{def:MB}
A map $g: \mathcal{S}_{++}^p \to [0,\infty)$ is called a matrix-variate Bernstein function with triplet $(A, B, \mu)$ if there exist a measure $\mu$ on $\mathcal{S}_{++}^p$ with \vspace{1mm}
\[
\smash{\int_{\mathcal{S}_{++}^p} \min\big[1, \{\mathrm{tr}(X^2)\}^{1/2}\big] \, \mu(\rd X) < \infty} \vspace{2mm}
\]
and real matrices $A, B\in \mathcal{S}_+^p$ such that, for all $T\in \mathcal{S}_{++}^p$,
\[
g(T) = \mathrm{tr}(A) + \mathrm{tr}(B T) + \int_{\mathcal{S}_{++}^p} \{1 - \etr(-T X)\} \, \mu(\rd X).
\]
Additionally, if $\mathfrak{X}$ is a $p \times p$ positive definite random matrix with induced probability measure $\mu$ on $\mathcal{S}_{++}^p$, then one can write, for all $T \in \mathcal{S}_{++}^p$,
\[
g(T) = \mathrm{tr}(A) + \mathrm{tr}(B T) + 1 - \EE\{\etr(-T \mathfrak{X})\} .
\]
\end{definition}

\begin{remark}\label{rem:Bernstein.cumulant}
According to \citet[pp.~156--157]{MR1155400} or \citet[Proposition~2.1]{MR2001835}, Definition~\ref{def:MB} is equivalent to the statement that the map $T\mapsto \mathrm{tr}(A) - g(T)$ is the cumulant transform of an infinitely divisible matrix distribution supported on the space of positive definite matrices $\mathcal{S}_{++}^p$.
\end{remark}

One way to construct a matrix-variate completely monotone function is through the composition of a (univariate) completely monotone function with a matrix-variate Bernstein function, as shown in the following proposition.

\begin{proposition}\label{prop:composition.MB.CM}
Let $g: \mathcal{S}_{++}^p \to [0,\infty)$ be a matrix-variate Bernstein function with triplet $(0_{p\times p},B,\mu)$, where $\mu$ is assumed to be a probability measure on $\mathcal{S}_{++}^p$. If the map $\phi : [0,\infty) \to [0,\infty)$ is continuous, completely monotone on $(0,\infty)$ and satisfies $\phi(0) = 1$, then the composition $\psi = \phi \circ g : \mathcal{S}_{++}^p \to [0,\infty)$ is a matrix-variate completely monotone function in the sense of Definition~\ref{def:MCM}.
\end{proposition}

Next, the notion of logarithmically complete monotonicity, as given in Definition~1 of \citet{MR2075188}, is extended to the matrix-variate setting; cf.\ Definition~5.10 and Theorem~5.11 of \citet{MR2978140}.

\begin{definition}%[Matrix-variate logarithmically complete monotonicity]
\label{def:MLCM}
A function $\psi : \mathcal{S}_{++}^p \to (0,1]$ is called matrix-variate logarithmically completely monotone if the negative of the logarithm of $\psi$, $-\ln  (\psi) : \mathcal{S}_{++}^p \to [0,\infty)$, is a matrix-variate Bernstein function in the sense of Definition~\ref{def:MB}.
\end{definition}

In the same way that the notion of logarithmically complete monotonicity on $(0,\infty)$ implies complete monotonicity on $(0,\infty)$, the fact that a function is matrix-variate logarithmically completely monotone on $\mathcal{S}_{++}^p$ implies that it is matrix-variate completely monotone on $\mathcal{S}_{++}^p$, as the following corollary states. The proof is an immediate consequence of Proposition~\ref{prop:composition.MB.CM} with the choice of functions $\phi(t) = \exp(-t)$ and $g(T) = -\ln \{\psi(T)\}$.

\begin{corollary}\label{cor:MLCM.implies.MCM}
If a function $\psi : \mathcal{S}_{++}^p \to (0,1]$ is matrix-variate logarithmically completely monotone, then it is matrix-variate completely monotone.
\end{corollary}

The following result is an easy consequence of Definition~\ref{def:MLCM} and Corollary~\ref{cor:MLCM.implies.MCM}.

\begin{corollary}\label{cor:exponential.MCM.implies.MLCM}
If a map $g : \mathcal{S}_{++}^p \to [0,\infty)$ is a matrix-variate Bernstein function, then $\psi = \exp(-g) : \mathcal{S}_{++}^p \to (0,1]$ is matrix-variate logarithmically completely monotone. In particular, by Corollary~\ref{cor:MLCM.implies.MCM}, it is also matrix-variate completely monotone.
\end{corollary}

The definition below extends to the matrix-variate setting the notion of multivariate Laplace transform order given, e.g., in Section~7.D.1 of~\citet{MR2265633}.

\begin{definition}%[Matrix-variate Laplace transform order]
\label{def:MLtO}
Let $n \in \N$ and $q_1, \ldots, q_n \in \N$ be given positive integers. Let also $(\mathfrak{X}_1, \ldots, \mathfrak{X}_n)$ and $(\mathfrak{Y}_1, \ldots, \mathfrak{Y}_n)$ be two random vectors of matrices supported on the space $\mathcal{S}_+^{q_1} \times \dots \times \mathcal{S}_+^{q_n}$. Suppose that for all $T_i \in \mathcal{S}_+^{q_i}$ and $i \in \{1, \ldots, n\}$, one has
\[
\EE\left\{\prod_{i=1}^n \etr(- T_i \mathfrak{X}_i)\right\} \geq \EE\left\{\prod_{i=1}^n \etr(- T_i \mathfrak{Y}_i)\right\}.
\]
Then $(\mathfrak{X}_1, \ldots, \mathfrak{X}_n)$ is said to be smaller or equal to $(\mathfrak{Y}_1, \ldots, \mathfrak{Y}_n)$ in the matrix-variate Laplace transform order, which is denoted by $(\mathfrak{X}_1, \ldots, \mathfrak{X}_n) \preceq_{\mathrm{Lt}} (\mathfrak{Y}_1, \ldots, \mathfrak{Y}_n)$.
\end{definition}

Next, recall the definition of multivariate total positivity of order~2 ($\mathrm{MTP}_2$), as formulated, e.g., by \citet[p.~1036]{MR628759}.

\begin{definition}\label{def:MTP2}
Let $d\in \N$ and $A_1,\ldots,A_d \subseteq \mathbb{R}$ be given. A probability density function $f: \mathbb{R}^d \to [0,\infty)$ supported on $A = A_1 \times \dots \times A_d$ is called multivariate totally positive of order $2$, denoted $\mathrm{MTP}_2$, if for all vectors $\bb{x} = (x_1, \ldots, x_d), \bb{y} = (y_1, \ldots, y_d)\in A$, one has
\[
f(\bb{x} \vee \bb{y}) f(\bb{x} \wedge \bb{y}) \geq f(\bb{x}) f(\bb{y}),
\]
where $\bb{x} \vee \bb{y} = (\max(x_1,y_1), \ldots, \max(x_d, y_d))$ and $\bb{x} \wedge \bb{y} = (\min(x_1,y_1), \ldots, \min(x_d, y_d))$ are com\-ponent-wise maxima and minima, respectively.
\end{definition}

\section{Main results}\label{sec:main.results}

The proposition below is an extension of Theorem~7.D.6 of \citet{MR2265633} from the multivariate Laplace transform order to the matrix-variate Laplace transform order of Definition~\ref{def:MLtO}. The result shows that the new matrix-variate Laplace transform order can be leveraged to obtain inequalities for product moments of matrix-variate completely monotone functions of certain positive definite random matrices.

\begin{proposition}\label{prop:7.D.6.Shaked}
For any integers $n, q_1, \ldots, q_n \in \N$, let $(\mathfrak{X}_1, \ldots, \mathfrak{X}_n)$ and $(\mathfrak{Y}_1, \ldots, \mathfrak{Y}_n)$ be two random vectors supported on the space $\mathcal{S}_{++}^{q_1} \times \dots \times \mathcal{S}_{++}^{q_n}$. Then, one has
\[
(\mathfrak{X}_1, \ldots, \mathfrak{X}_n) \preceq_{\mathrm{Lt}} (\mathfrak{Y}_1, \ldots, \mathfrak{Y}_n) \quad \Leftrightarrow \quad \EE\left\{\prod_{i=1}^n \phi_i(\mathfrak{X}_i)\right\} \geq \EE\left\{\prod_{i=1}^n \phi_i(\mathfrak{Y}_i)\right\}
\]
for all matrix-variate completely monotone functions $\phi_i : \mathcal{S}_{++}^{q_i} \to [0,\infty)$ in the sense of Definition~\ref{def:MCM}, whenever the expectations are finite.
\end{proposition}

As an example, the next proposition describes how vectors of disjoint diagonal blocks of Wishart random matrices are smaller than, or equal to, their independent counterparts in the matrix-variate Laplace transform order.

\begin{proposition}\label{prop:Wishart.diag.block.Lt.order}
For any integers $d, p_1, \ldots, p_d \in \N$, let $p = p_1 + \cdots + p_d$ and fix any $\alpha \in (p-1, \infty)$. Further let $\Sigma\in \mathcal{S}_{++}^p$ be any $p \times p$ positive definite matrix expressed in the form $\Sigma = (\Sigma_{ij})$, where for every integers $i, j \in \{ 1, \ldots, d \}$, the block $\Sigma_{ij}$ has size $p_i \times p_j$. For any given integer $k\in \{2,\ldots,d\}$, let also
\[
\mathfrak{X} \sim\mathcal{W}_p(\alpha,\Sigma), \quad \mathfrak{X}^{\star} \sim\mathcal{W}_p(\alpha, \Sigma_{1:k-1,1:k-1} \midoplus \Sigma_{k:d,k:d})
\]
be partitioned conformally with $\Sigma$ (meaning that $\mathfrak{X}_{ij}$ and $\mathfrak{X}^{\star}_{ij}$ have the same size as $\Sigma_{ij}$ for every $i, j \in \{ 1, \ldots, d \}$). Then, in the sense of Definition~\ref{def:MLtO}, one has both
\[
\bigoplus_{i=1}^d \mathfrak{X}_{ii} \preceq_{\mathrm{Lt}} \bigoplus_{i=1}^d \mathfrak{X}_{ii}^{\star} \quad \text{and} \quad (\mathfrak{X}_{11}, \ldots, \mathfrak{X}_{dd}) \preceq_{\mathrm{Lt}} (\mathfrak{X}^{\star}_{11}, \ldots, \mathfrak{X}^{\star}_{dd}).
\]
\end{proposition}

The following result extends in two ways Theorem~3.2 of \citet{MR3278931}, which is listed in item~\ref{item:n} of Section~\ref{sec:intro}. First, the disjoint diagonal blocks $\mathfrak{X}_{ii}$ are $p_i\times p_i$ for each $i\in \{1, \ldots, d\}$, instead of only $1\times 1$ --- recall that multivariate Gamma random vectors in the sense of \citet{MR44790} are diagonal elements of Wishart random matrices. Second, matrix-variate completely monotone functions $\phi_i : \mathcal{S}_{++}^{p_i} \to [0,\infty)$ replace the inverse power functions $x_i\mapsto x_i^{-\nu_i}$ for $i \in \{ 1, \ldots, d \}$.

\begin{theorem}\label{thm:generalization.Theorem.3.2.Wei}
For any integers $d, p_1, \ldots, p_d \in \N$, let $p = p_1 + \cdots + p_d$ and fix any $\alpha \in (p-1, \infty)$. Further let $\Sigma\in \mathcal{S}_{++}^p$ be any $p \times p$ positive definite matrix expressed in the form $\Sigma = (\Sigma_{ij})$, where for every integers $i, j \in \{ 1, \ldots, d \}$, the block $\Sigma_{ij}$ has size $p_i \times p_j$. Moreover, for each integer $i \in \{ 1, \ldots, d\}$, let $\phi_i : \mathcal{S}_{++}^{p_i} \to [0,\infty)$ be an arbitrary matrix-variate completely monotone function. Let $M$ be the unique block lower triangular matrix of size $p\times p$ defined through the block Cholesky decomposition $\Sigma = M M^{\top}$, where the diagonal block $M_{ii}$ is chosen to be positive definite for each $i \in \{1,\ldots,d\}$. For any $p \times p$ Wishart random matrix $\mathfrak{X}\sim\mathcal{W}_p (\alpha,\Sigma)$ partitioned conformally with~$\Sigma$, and any given integer $k\in \{2,\ldots,d\}$, there exists a product measure $\mu_1 \times \dots \times \mu_d$ on $\mathcal{S}_+^{p_1} \times \dots \times \mathcal{S}_+^{p_d}$ such that
\[
\EE\left\{\prod_{i=1}^{k-1} \phi_i(\mathfrak{X}_{ii})\right\} \EE\left\{\prod_{i=k}^d \phi_i(\mathfrak{X}_{ii})\right\}
\leq \EE\left\{\prod_{i=1}^d \phi_i(\mathfrak{X}_{ii})\right\}
\leq \prod_{i=1}^d \int_{\mathcal{S}_+^{p_i}} \frac{2^{p_i \alpha/2} \mu_i(\rd T_{ii})}{\big|I_{p_i} + \sqrt{2} \, T_{ii}^{1/2} M_{ii}\big|^{\alpha}},
\]
provided that the expectations and integrals are finite.
\end{theorem}

Here is an example of a matrix-variate completely monotone function, among others. Using the fact that the $\mathcal{W}_{p_i}(2\nu_i,T^{-1}/2)$ density function integrates to $1$ on $\mathcal{S}_{++}^{p_i}$, it is straightforward to show that for any integer $i \in \{1, \ldots, d\}$, $\nu_i > (p_i - 1)/2$, and positive definite matrix $T \in \mathcal{S}_{++}^{p_i}$, one has
\begin{equation}\label{eq:3.1}
|T|^{-\nu_i} = \int_{\mathcal{S}_{++}^{p_i}} \etr(-T X) \, \frac{|X|^{\nu_i - (p_i + 1)/2}}{\Gamma_{p_i}(\nu_i)} \, \rd X = \mathcal{L}(\mu_i)(T),
\end{equation}
where the measures $\mu_i$ on $\mathcal{S}_{++}^{p_i}$ are defined by
\[
\mu_i(\rd X) = \frac{|X|^{\nu_i - (p_i + 1)/2}}{\Gamma_{p_i}(\nu_i)} \, \rd X.
\]
Equation~\eqref{eq:3.1} implies that, for $\nu_i > (p_i - 1)/2$, the map $T\mapsto |T|^{-\nu_i}$ is matrix-variate completely monotone on $\mathcal{S}_{++}^{p_i}$ as per Definition~\ref{def:MCM}. Therefore, as an application of Theorem~\ref{thm:generalization.Theorem.3.2.Wei}, the following corollary ensues.

\begin{corollary}\label{cor:GPI.negative.powers.determinants.Wishart}
In the setting of Theorem~\ref{thm:generalization.Theorem.3.2.Wei}, one has
\[
\begin{aligned}
\EE\left(\prod_{i=1}^{k-1} |\mathfrak{X}_{ii}|^{-\nu_i}\right) \EE\left(\prod_{i=k}^d |\mathfrak{X}_{ii}|^{-\nu_i}\right)
&\leq \EE\left(\prod_{i=1}^d |\mathfrak{X}_{ii}|^{-\nu_i}\right) \\
&\leq \prod_{i=1}^d \frac{2^{p_i \alpha/2}}{\Gamma_{p_i}(\nu_i)} \int_{\mathcal{S}_{++}^{p_i}} \frac{|T_{ii}|^{\nu_i - (p_i + 1)/2}}{\big|I_{p_i} + \sqrt{2} \, T_{ii}^{1/2} M_{ii}\big|^{\alpha}} \, \rd T_{ii},
\end{aligned}
\]
provided that $(p_i - 1)/2 < \nu_i < \alpha/2 - (p_i - 1)/2$ for all $i\in \{1,\ldots,d\}$.
\end{corollary}

\begin{remark}\label{rem:finite.exp}
In Corollary~\ref{cor:GPI.negative.powers.determinants.Wishart}, the condition $(p_i - 1)/2 < \nu_i$ for all $i\in \{1,\ldots,d\}$ is sufficient for the maps $T\mapsto |T|^{-\nu_i}$ to be matrix-variate completely monotone and thus for the application of Theorem~\ref{thm:generalization.Theorem.3.2.Wei}. Lemma~\ref{lem:finite.exp} of \ref{app:B} shows that a sufficient condition for the finiteness of the expectations is $(p_i - 1)/2 < \nu_i < \alpha/2 - (p_i - 1)/2$ for all $i\in \{1,\ldots,d\}$, while the condition $\nu_i < \alpha/2 - (p_i - 1)/2$ for all $i\in \{1,\ldots,d\}$ is necessary. \ref{app:C} shows that the integrals on the right-hand side are finite if and only if $(p_i - 1)/2 < \nu_i < \alpha/2 - p_i (p_i + 1)/4$ for all $i\in \{1,\ldots,d\}$. In particular, these restrictions are coherent as it is always true that $\alpha/2 - p_i (p_i + 1)/4 < \alpha/2 - (p_i - 1)/2$.
\end{remark}

\begin{remark}
The lower bound in Corollary~\ref{cor:GPI.negative.powers.determinants.Wishart} extends Corollary~1~(a) in \cite{MR4538422} for traces of disjoint diagonal blocks of Wishart random matrices to determinants of such blocks. The lower and upper bounds together extend significantly Theorem~3.2 of \citet{MR3278931}, where the same two inequalities were proved for diagonal elements of Wishart random matrices (i.e., $p_1 = \dots = p_d = 1$ or $1 \times 1$ diagonal blocks), which are known to follow the multivariate Gamma distribution in the sense of \citet{MR44790}. This distribution was studied extensively by Royen; see, e.g., \citet{MR3325368} for a survey of its properties.
\end{remark}

\begin{remark}
In order to render the upper bound in Corollary~\ref{cor:GPI.negative.powers.determinants.Wishart} practicable, explicit expressions for the integrals are derived in \ref{app:C}.
\end{remark}

As explained in Remark~4.2 of \citet{MR4554766}, an alternative way of showing the lower bound in Corollary~\ref{cor:GPI.negative.powers.determinants.Wishart} would be to prove a Gaussian correlation-type inequality for disjoint principal minors of Wishart random matrices. For the aforementioned multivariate Gamma distribution, this Gaussian correlation inequality was proved by Royen in 2014; see \citet{MR3289621} and also \citet{MR3468024,arXiv:2006.00769} for some extensions. Consequently, it is natural to conjecture (see Conjecture~\ref{conj:1} below) that the same inequality holds more broadly for disjoint diagonal blocks of any widths $p_1, \ldots, p_d\in \N$.

\begin{conjecture}\label{conj:1}
In the setting of Theorem~\ref{thm:generalization.Theorem.3.2.Wei}, one has, for every $k\in \{2,\ldots,d\}$ and $(t_1, \ldots, t_d)\in (0,\infty)^d$,
\[
\PP\left(\bigcap_{i=1}^d \{|\mathfrak{X}_{ii}| \leq t_i\}\right) \geq \PP\left(\bigcap_{i=1}^{k-1} \{|\mathfrak{X}_{ii}| \leq t_i\}\right) \PP\left(\bigcap_{i=k}^d \{|\mathfrak{X}_{ii}| \leq t_i\}\right).
\]
\end{conjecture}

The proposition below is a generalization of Theorem~1.1 of \citet{MR4666255} from the multivariate normal setting to the case of disjoint principal minors of Wishart random matrices. It is part of the realm of so-called `opposite GPIs' introduced by \citet{MR4471184}, which are described in Section~\ref{sec:intro}. Note that the validity of the proposition is conditional on Conjecture~\ref{conj:1.1} being true. \citet{MR4666255} proved the result below unconditionally in the multivariate normal setting, but only when $\nu_2 = \cdots = \nu_d = 1$, by exploiting a GPI result of \citet{MR2385646}; see item~\ref{item:b} in Section~\ref{sec:intro}.

\begin{proposition}\label{prop:analogue.Theorem.1.1.Zhou.et.al}
For any integers $d, p_1, \ldots, p_d \in \N$, let $p = p_1 + \cdots + p_d$ and fix any $\alpha \in (p-1, \infty)$. Further let $\Sigma\in \mathcal{S}_{++}^p$ be any $p \times p$ positive definite matrix expressed in the form $\Sigma = (\Sigma_{ij})$, where for every integers $i, j \in \{ 1, \ldots, d \}$, the block $\Sigma_{ij}$ has size $p_i \times p_j$. For each $i\in \{1,\ldots,d\}$, define $P_{1i} = \smash{\Sigma_{11}^{-1/2} \Sigma_{1i} \Sigma_{ii}^{-1/2}}$. Let $\mathfrak{X}\sim\mathcal{W}_p (\alpha,\Sigma)$ be a $p \times p$ Wishart random matrix partitioned conformally with $\Sigma$. Then conditionally on Conjecture~\ref{conj:1.1} being true, one has
\[
\EE\left(|\mathfrak{X}_{11}|^{-\nu_1} \prod_{i=2}^d |\mathfrak{X}_{ii}|^{\nu_i}\right) \geq \left(\prod_{i=2}^d |I_{p_i} - P_{1i}^{\top} P_{1i}|^{\nu_i}\right) \EE(|\mathfrak{X}_{11}|^{-\nu_1}) \prod_{i=2}^d \EE(|\mathfrak{X}_{ii}|^{\nu_i}),
\]
for any $(p_1 - 1)/2 < \nu_1 < \alpha/2 - (p_1 - 1)/2$ and $\nu_2, \ldots, \nu_d \in (0,\infty)$.
\end{proposition}

\begin{remark}
Given that Conjecture~\ref{conj:1.1} was proved to hold for $d = 2$ in \cite{arXiv:2409.14512}, then Proposition~\ref{prop:analogue.Theorem.1.1.Zhou.et.al} is true unconditionally when $d=2$.
\end{remark}

The proposition below is a generalization of Theorem~1.3 of \citet{MR4666255} from the multivariate normal setting to the case of disjoint principal minors of Wishart random matrices. The proposition is unconditional and it is also part of the realm of so-called `opposite GPIs' mentioned above.

\begin{proposition}\label{prop:analogue.Theorem.1.3.Zhou.et.al}
For any integers $d, p_1, \ldots, p_d \in \N$, let $p = p_1 + \cdots + p_d$ and fix any $\alpha \in (p-1, \infty)$. Further let $\Sigma\in \mathcal{S}_{++}^p$ be any $p \times p$ positive definite matrix expressed in the form $\Sigma = (\Sigma_{ij})$, where for every integers $i, j \in \{ 1, \ldots, d \}$, the block $\Sigma_{ij}$ has size $p_i \times p_j$. Let $\mathfrak{X}\sim\mathcal{W}_p (\alpha,\Sigma)$ be a $p \times p$ Wishart random matrix partitioned conformally with $\Sigma$. Then
\[
\EE\left\{\left(\prod_{i=1}^{d-1} |\mathfrak{X}_{ii}|^{-\nu_i}\right) |\mathfrak{X}_{dd}|^{\nu_d}\right\} \leq \EE\left(\prod_{i=1}^{d-1} |\mathfrak{X}_{ii}|^{-\nu_i}\right) \EE\left(|\mathfrak{X}_{dd}|^{\nu_d}\right),
\]
provided that $(p_i - 1)/2 < \nu_i < \alpha/2 - (p_i - 1)/2$ for all $ i\in \{1, \ldots, d-1\}$ and $\nu_d\in (0,\infty)$. The restrictions are analogous to those of Corollary~\ref{cor:GPI.negative.powers.determinants.Wishart} and explained in Remark~\ref{rem:finite.exp}.
\end{proposition}

The following result extends Theorem~2 in \cite{MR4538422} from the setting of traces of disjoint diagonal blocks of Wishart random matrices to the case of determinants of such blocks.

\begin{theorem}\label{thm:generalization.Theorem.2.GO.2023}
Let $p_1,p_2\in \N$ be given integers. Let $(\mathfrak{X}_{11}, \mathfrak{X}_{22})$ and $(\mathfrak{X}_{11}^{\star}, \mathfrak{X}_{22}^{\star})$ be two random pairs on the space $\mathcal{S}_{++}^{p_1} \times \mathcal{S}_{++}^{p_2}$ with identical marginal distributions. Suppose that the random matrices $\mathfrak{X}_{11}^{\star}$ and $\mathfrak{X}_{22}^{\star}$ are independent, and that $(\mathfrak{X}_{11}, \mathfrak{X}_{22}) \preceq_{\mathrm{Lt}} (\mathfrak{X}_{11}^{\star}, \mathfrak{X}_{22}^{\star})$. If, as per Definition~\ref{def:MB}, $f$ and $g$ are matrix-variate Bernstein functions with triplets $(A_1,0,\mu_1)$ and $(A_2,0,\mu_2)$, respectively, then
\[
\EE\{f(\mathfrak{X}_{11}) g(\mathfrak{X}_{22})\} \geq \EE\{f(\mathfrak{X}_{11}^{\star})\} \EE\{g(\mathfrak{X}_{22}^{\star})\},
\]
provided that the expectations are finite.
\end{theorem}

\begin{remark}
In order to apply Theorem~\ref{thm:generalization.Theorem.2.GO.2023} with specific examples of matrix-variate Bernstein functions, recall Remark~\ref{rem:Bernstein.cumulant} regarding the relationship between matrix-variate Bernstein functions and cumulant transforms of infinitely divisible matrix distributions supported on $\mathcal{S}_{++}^p$ and refer to \citet{MR2001835,MR2354571} for appropriate examples of cumulant transforms in this context.
\end{remark}

The next result differs in nature from all previous ones, but it is relevant for the current setting as it shows that the eigenvalues of a Wishart random matrix satisfy a more general form of GPI.

\begin{theorem}\label{thm:eigenvalues.MTP2}
Fix any integers $n, p \in \N$, $\alpha \in (p - 1,\infty)$, and $p \times p$ positive definite matrix $\Sigma\in \mathcal{S}_{++}^p$. Let $\mathfrak{X} \sim\mathcal{W}_p (\alpha,\Sigma)$ be a $p \times p$ Wishart random matrix, and let $\bb{\Lambda} = (\Lambda_1, \ldots, \Lambda_p)$ denote its vector of eigenvalues. Then, for any nonnegative and component-wise non-decreasing functions $h_1, \ldots, h_n$ supported on $(0,\infty)^p$, and for every integer $k \in \{2, \ldots, n\}$, one has
\[
\EE \left\{ \prod_{i=1}^n h_i(\bb{\Lambda}) \right\} \geq \EE \left\{ \prod_{i=1}^{k - 1} h_i(\bb{\Lambda}) \right\} \EE \left\{ \prod_{i = k}^n h_i(\bb{\Lambda}) \right\},
\]
provided that the expectations are finite. Moreover, the inequality also holds if the functions $h_1, \ldots, h_n$ are all component-wise non-increasing.
\end{theorem}

In particular, for any $\nu\in [0,\infty)$, the power function $x \mapsto x^{\nu}$ is non-decreasing on $(0,\infty)$, so the corollary below is an immediate consequence of Theorem~\ref{thm:eigenvalues.MTP2} with $n = p$ and the choice of functions $h_i(\bb{\lambda}) = \lambda_i^{\nu_i}$ for each $i \in \{1,\ldots, p\}$.

\begin{corollary}\label{cor:eigenvalues.MTP2.powers}
Fix any integer $p \in \N$, $\alpha \in (p - 1,\infty)$, and $p \times p$ positive definite matrix $\Sigma\in \mathcal{S}_{++}^p$. Let $\mathfrak{X} \sim\mathcal{W}_p (\alpha,\Sigma)$ be a $p \times p$ Wishart random matrix, and let $\bb{\Lambda} = (\Lambda_1, \ldots, \Lambda_p)$ denote its vector of eigenvalues. Then, for every $\nu_1, \ldots, \nu_p\in [0,\infty)$ and every integer $k \in \{2, \ldots, p\}$, one has
\[
\EE \left( \prod_{i=1}^p \Lambda_i^{\nu_i} \right)
\geq \EE \left( \prod_{i=1}^{k-1} \Lambda_i^{\nu_i} \right) \EE \left( \prod_{i=k}^p \Lambda_i^{\nu_i} \right).
\]
\end{corollary}

\begin{remark}
To see that the expectations in Corollary~\ref{cor:eigenvalues.MTP2.powers} are finite, let $\nu = \nu_1 + \dots + \nu_p$. Then, the weighted arithmetic mean--geometric mean (AM--GM) inequality yields
\[
\EE \left( \prod_{i=1}^p \Lambda_i^{\nu_i} \right)
\leq \EE \left\{ \left(\sum_{i=1}^p \frac{\nu_i}{\nu} \Lambda_i\right)^{\nu} \right\} \leq \EE \left\{ \left(\sum_{i=1}^p \Lambda_i\right)^{\nu} \right\}
= \EE \left[ \{\mathrm{tr}(\mathfrak{X})\}^{\nu} \right].
\]
The trace moments on the right-hand side are known to be finite for every $\nu\in [0,\infty)$; see, e.g., Corollary~7.2.8 of \cite{MR652932} together with the bottom of p.~245 in the same book.
\end{remark}

\section{Proofs}\label{sec:proofs}

\begin{proof}[\textbf{Proof of Proposition~\ref{prop:composition.MB.CM}}]
By the assumptions, the map $g: \mathcal{S}_{++}^p \to [0,\infty)$ is a matrix-variate Bernstein function with triplet $(0_{p\times p},B,\mu)$, where $\mu$ is a probability measure on $\mathcal{S}_{++}^p$. Therefore, if a real random matrix $\mathfrak{X}$ of size $p\times p$ has law $\mu$, one can write
\[
g(T) = \mathrm{tr}(B T) + 1 - \EE\{\etr(-T \mathfrak{X})\}.
\]
Moreover, the function $\phi : [0,\infty) \to [0,\infty)$ is assumed to be completely monotone on $(0,\infty)$, so Bernstein's theorem, recall~\eqref{eq:2.1}, states that there exists a measure $\nu$ on $[0,\infty)$ such that, for all $t \in (0,\infty)$, one has
\[
\phi(t) = \int_{[0,\infty)} e^{-t\lambda} \nu (\rd \lambda).
\]
Because it is further assumed that $\phi$ is continuous and $\phi(0) = 1$, one knows that $\nu$ is in fact a probability measure; see, e.g., Section~XIII.4 of \citet{MR270403}.

Hence, letting $\Lambda\sim \nu$ and $N \mid \Lambda \sim \mathcal{P}\hspace{0.2mm}(\Lambda)$ be independent from a mutually independent sequence $\mathfrak{X}_1,\mathfrak{X}_2, \ldots \sim \mu$, with $\mathcal{P}$ referring to the Poisson distribution, one can write
\[
\begin{aligned}
\psi(T) = \phi\{g(T)\}
&= \int_{[0,\infty)} \etr(-\lambda B T) \exp(-\lambda) \exp\left[\lambda \EE\{\etr(-T \mathfrak{X})\}\right] \nu(\rd \lambda) \\
&= \int_{[0,\infty)} \etr(-\lambda B T) \sum_{n=0}^{\infty} \big[\EE\{\etr(-T \mathfrak{X})\}\big]^n \frac{\lambda^n}{n!} \exp(-\lambda) \, \nu(\rd \lambda) \\
&= \int_{[0,\infty)} \etr(-\lambda B T) \, \EE\left\{\etr\left(-T \sum_{i=1}^N \mathfrak{X}_i\right) \mid \Lambda = \lambda\right\} \nu(\rd \lambda) \\
&= \int_{[0,\infty)} \EE\left\{\etr\left(-T B \Lambda - T \sum_{i=1}^N \mathfrak{X}_i\right) \mid \Lambda = \lambda\right\} \nu(\rd \lambda) \\
&= \EE\left[\etr\left\{-T \left(B \Lambda + \sum_{i=1}^N \mathfrak{X}_i\right)\right\}\right],
\end{aligned}
\]
which is the Laplace transform of the positive definite random matrix $B \Lambda + \mathfrak{X}_1 + \cdots + \mathfrak{X}_N$. Therefore, $\psi = \phi \circ g : \mathcal{S}_{++}^p \to [0,\infty)$ is a matrix-variate completely monotone function, as per Definition~\ref{def:MCM}.
\end{proof}

\begin{proof}[\textbf{Proof of Proposition~\ref{prop:7.D.6.Shaked}}]
Let $\phi_1, \ldots, \phi_n$ be a collection of matrix-variate completely monotone functions such that $\phi_i: \mathcal{S}_{++}^{q_i}\to [0,\infty)$ for each integer $i \in \{1, \ldots, n\}$. By Definition~\ref{def:MCM}, for each $i\in \{1, \ldots, n\}$, there exists a measure $\mu_i$ on $\mathcal{S}_+^{q_i}$ such that $\phi_i = \mathcal{L}(\mu_i)$. Therefore, Fubini's theorem and the assumption that $(\mathfrak{X}_1, \ldots, \mathfrak{X}_n) \preceq_{\mathrm{Lt}} (\mathfrak{Y}_1, \ldots, \mathfrak{Y}_n)$ together imply
\[
\begin{aligned}
\EE\left\{\prod_{i=1}^n \phi_i(\mathfrak{X}_i)\right\}
&= \EE\left\{\int_{\mathcal{S}_+^{q_1}} \dots\int_{\mathcal{S}_+^{q_n}} \prod_{i=1}^n \etr(- \mathfrak{X}_i S_i) \, \mu_1 (\rd S_n) \cdots \mu_1 (\rd S_1) \right\} \\
&= \int_{\mathcal{S}_+^{q_1}} \dots\int_{\mathcal{S}_+^{q_n}} \EE\left\{\prod_{i=1}^n \etr(- \mathfrak{X}_i S_i)\right\} \mu_n (\rd S_n) \cdots \mu_1 (\rd S_1) \\
&\geq \int_{\mathcal{S}_+^{q_1}} \dots\int_{\mathcal{S}_+^{q_n}} \EE\left\{\prod_{i=1}^n \etr(- \mathfrak{Y}_i S_i)\right\} \mu_n (\rd S_n) \cdots \mu_1 (\rd S_1) \\
&= \EE\left\{\int_{\mathcal{S}_+^{q_1}} \dots\int_{\mathcal{S}_+^{q_n}} \prod_{i=1}^n \etr(- \mathfrak{Y}_i S_i) \, \mu_n (\rd S_n) \cdots \mu_1 (\rd S_1)\right\} \\
&= \EE\left\{\prod_{i=1}^n \phi_i(\mathfrak{Y}_i)\right\}.
\end{aligned}
\]
This proves the `only if' part of the statement.

The `if' part is even simpler. For each integer $i \in \{1, \ldots, n\}$ and any given $T_i\in \mathcal{S}_+^{q_i}$, simply take $\phi_i: \mathcal{S}_{++}^{q_i}\to [0,\infty)$ to be the matrix-variate completely monotone function $X_i \mapsto \etr(-T_i X_i)$ in
\[
\EE\left\{\prod_{i=1}^n \phi_i(\mathfrak{X}_i)\right\} \geq \EE\left\{\prod_{i=1}^n \phi_i(\mathfrak{Y}_i)\right\}.
\]
To verify that $X_i \mapsto \etr(-T_i X_i)$ is indeed matrix-variate completely monotone, apply Definition~2.1 of \citet{MR3286037}.
The above shows that $(\mathfrak{X}_1, \ldots, \mathfrak{X}_n) \preceq_{\mathrm{Lt}} (\mathfrak{Y}_1, \ldots, \mathfrak{Y}_n)$ as per Definition~\ref{def:MLtO} and concludes the proof.
\end{proof}

\newpage
\begin{proof}[\textbf{Proof of Proposition~\ref{prop:Wishart.diag.block.Lt.order}}]
Let $T\in \mathcal{S}_+^p$ be a $p \times p$ nonnegative definite matrix, and assume that $T$ is partitioned conformally with $\Sigma$. By using the expression for the Laplace transform of the Wishart distribution in Remark~\ref{rem:Wishart.Lt}, one has
\[
\EE\left\{\etr\left(-T \bigoplus_{i=1}^d \mathfrak{X}_{ii}\right)\right\}
= \EE\left[\etr\left\{-\left(\bigoplus_{i=1}^d T_{ii}\right) \mathfrak{X}\right\}\right]
= \left|I_p + 2 \left(\bigoplus_{i=1}^d T_{ii}\right) \Sigma\right|^{-\alpha/2}.
\]
Likewise, for any given integer $k\in \{2,\ldots,d\}$, one has
\[
\begin{aligned}
\EE\left\{\etr\left(-T \bigoplus_{i=1}^d \mathfrak{X}_{ii}^{\star}\right)\right\}
&= \left|I_p + 2 \left(\bigoplus_{i=1}^d T_{ii}\right) \left(\Sigma_{1:k-1,1:k-1} \midoplus \Sigma_{k:d,k:d}\right)\right|^{-\alpha/2} \\
&= \left|I_{\sum_{i=1}^{k-1} p_i} + 2 \left(\bigoplus_{i=1}^{k-1} T_{ii}\right) \Sigma_{1:k-1,1:k-1}\right|^{-\alpha/2} \\
&\qquad\times \left|I_{\sum_{i=k}^d p_i} + 2 \left(\bigoplus_{i=k}^d T_{ii}\right) \Sigma_{k:d,k:d}\right|^{-\alpha/2}.
\end{aligned}
\]
By comparing the right-hand sides of the last two equations, an application of Fischer's inequality (see, e.g., Theorem~7.8.5 of \citet{MR2978290}) yields
\[
\EE\left\{\etr\left(-T \bigoplus_{i=1}^d \mathfrak{X}_{ii}\right)\right\} \geq \EE\left\{\etr\left(-T \bigoplus_{i=1}^d \mathfrak{X}_{ii}^{\star}\right)\right\},
\]
and thus $\scalebox{1.2}{$\oplus$}_{i=1}^d \mathfrak{X}_{ii} \hspace{-0.5mm}\preceq_{\mathrm{Lt}} \hspace{-0.5mm}\scalebox{1.2}{$\oplus$}_{i=1}^d \mathfrak{X}_{ii}^{\star}$ as per Definition~\ref{def:MLtO}. By the linearity of the trace operator, the above can also be rewritten as
\[
\EE\left\{\prod_{i=1}^d \etr(- T_{ii} \mathfrak{X}_{ii})\right\} \geq \EE\left\{\prod_{i=1}^d \etr(- T_{ii} \mathfrak{X}^{\star}_{ii})\right\},
\]
which proves that $(\mathfrak{X}_{11}, \ldots, \mathfrak{X}_{dd}) \preceq_{\mathrm{Lt}} (\mathfrak{X}^{\star}_{11}, \ldots, \mathfrak{X}^{\star}_{dd})$ as per Definition~\ref{def:MLtO}.
\end{proof}

\begin{proof}[\textbf{Proof of Theorem~\ref{thm:generalization.Theorem.3.2.Wei}}]
The lower bound follows from Propositions~\ref{prop:7.D.6.Shaked}~and~\ref{prop:Wishart.diag.block.Lt.order} together with the fact that $(\mathfrak{X}_{11},\ldots,\mathfrak{X}_{k-1,k-1})$ and $(\mathfrak{X}_{11}^{\star},\ldots,\mathfrak{X}_{k-1,k-1}^{\star})$ have the same distribution, respectively, $(\mathfrak{X}_{kk},\ldots,\mathfrak{X}_{dd})$ and $(\mathfrak{X}_{kk}^{\star},\ldots,\mathfrak{X}_{dd}^{\star})$ have the same distribution, both in view of Corollary~3.2.6 of \citet{MR652932}.

For the upper bound, Definition~\ref{def:MCM} states that, for every integer $i \in \{1, \ldots, d\}$, there exists a measure $\mu_i$ on $\mathcal{S}_+^{p_i}$ such that, for every positive definite matrix $T \in \mathcal{S}_{++}^{p_i}$,
\[
\phi_i(T) = \mathcal{L}(\mu_i)(T) = \int_{\mathcal{S}_+^k} \etr(-T X) \, \mu_i(\rd X).
\]
Hence, by applying Fubini's theorem and the expression for the Laplace transform of the Wishart distribution in Remark~\ref{rem:Wishart.Lt}, one obtains
\[
\begin{aligned}
\EE\left\{\prod_{i=1}^d \phi_i(\mathfrak{X}_{ii})\right\}
&= \int_{\mathcal{S}_+^{p_1} \times \, \cdots \, \times \mathcal{S}_+^{p_d}} \EE\left\{\prod_{i=1}^d \etr\left(- T_{ii} \mathfrak{X}_{ii}\right)\right\} \prod_{i=1}^d \mu_i(\rd T_{ii}) \\
&= \int_{\mathcal{S}_+^{p_1} \times \, \cdots \, \times \mathcal{S}_+^{p_d}} \EE\left[\etr\left\{-\left(\bigoplus_{i=1}^d T_{ii}\right) \mathfrak{X}\right\}\right] \prod_{i=1}^d \mu_i(\rd T_{ii}) \\
&= \int_{\mathcal{S}_+^{p_1} \times \, \cdots \, \times \mathcal{S}_+^{p_d}} \left|I_p + 2 \left(\bigoplus_{i=1}^d T_{ii}\right) \Sigma\right|^{-\alpha/2} \prod_{i=1}^d \mu_i(\rd T_{ii}).
\end{aligned}
\]

In order to find an upper bound on the integrand, the reasoning for the upper bound in the proof of Theorem~3.2 of \citet{MR3278931} is adapted to the present setting. Let $V = \sqrt{2} \, L^{\top} M$ be a real-valued matrix of size $p \times p$, where the block lower triangular matrices $L$ and $M$ are defined through the block Cholesky decompositions
\[
\bigoplus_{i=1}^d T_{ii} = L L^{\top} \equiv \bigoplus_{i=1}^d T_{ii}^{1/2} T_{ii}^{1/2}, \qquad \Sigma = M M^{\top},
\]
and the diagonal blocks $L_{ii}$ and $M_{ii}$ are chosen to be positive definite for all $i\in \{1,\ldots,d\}$. One can then write
\[
\left|I_p + 2 \left(\bigoplus_{i=1}^d T_{ii}\right) \Sigma\right|^{-\alpha/2} = \big|I_p + V^{\top} V\big|^{-\alpha/2}.
\]
Now the matrix $(I_p + V^{\top} V) - (I_p + V)^{\top} (I_p + V) / 2 = (I_p - V)^{\top} (I_p - V) / 2$ is nonnegative definite, and the identity $|A_1 + A_2| \geq |A_2|$ holds for any nonnegative definite matrices $A_1,A_2\in \mathcal{S}_+^p$. Thus,
\[
\begin{aligned}
\big|I_p + V^{\top} V\big|^{-\alpha/2}
&= \big|(I_p + V^{\top} V) - (I_p + V)^{\top} (I_p + V) / 2 + (I_p + V)^{\top} (I_p + V) / 2\big|^{-\alpha/2} \\
&\leq \big|(I_p + V)^{\top} (I_p + V) / 2\big|^{-\alpha/2} \\[0.3mm]
&= 2^{p\alpha/2} |I_p + V|^{-\alpha}.
\end{aligned}
\]
Given that $L$ is block diagonal and $M$ is block lower triangular, the matrix $V$ is block lower triangular. Hence, the last quantity can be rewritten as
\[
\begin{aligned}
2^{p\alpha/2} |I_p + V|^{-\alpha}
&= \prod_{i=1}^d 2^{p_i \alpha/2} |I_{p_i} + V_{ii}|^{-\alpha} \\
&= \prod_{i=1}^d 2^{p_i \alpha/2} \big|I_{p_i} + \sqrt{2} \, T_{ii}^{1/2} M_{ii}\big|^{-\alpha}.
\end{aligned}
\]
By connecting the previous equations together, one reaches the desired conclusion.
\end{proof}

\begin{proof}[\textbf{Proof of Proposition~\ref{prop:analogue.Theorem.1.1.Zhou.et.al}}]
Fix numbers $\nu_1,\ldots,\nu_p$ where
\[
(p_1 - 1)/2 < \nu_1 < \alpha/2 - (p_1 - 1)/2
\]
and $\nu_i \in (0,\infty)$ for all $i \in \{2, \ldots, d\}$. Also, for any given positive definite matrix $T_{11}\in \mathcal{S}_{++}^{p_1}$, define $\overline{T}_{11} = T_{11} \midoplus 0_{(p-p_1)\times (p-p_1)}$ for convenience. Using the Laplace transforms in~\eqref{eq:3.1} for inverse determinant powers, one finds
\[
\EE\left(|\mathfrak{X}_{11}|^{-\nu_1} \prod_{i=2}^d |\mathfrak{X}_{ii}|^{\nu_i}\right)
= \EE\left\{\int_{\mathcal{S}_+^{p_1}} \etr\left(-T_{11} \mathfrak{X}_{11}\right) \, \mu_1(\rd T_{11}) \prod_{i=2}^d |\mathfrak{X}_{ii}|^{\nu_i}\right\}.
\]
Now using Fubini's theorem and the expression for the density function of the Wishart distribution in Definition~\ref{def:Wishart}, the right-hand side can be rewritten as
\begin{multline*}
 \int_{\mathcal{S}_+^{p_1}}\int_{\mathcal{S}_{++}^p} \left(\prod_{i=2}^d |X_{ii}|^{\nu_i}\right) \, \etr\left(-\overline{T}_{11} X\right) f_{\alpha,\Sigma}(X) \, \rd X \, \mu_1(\rd T_{11}) \\
= \int_{\mathcal{S}_+^{p_1}} |I_p + 2 \overline{T}_{11} \Sigma|^{-\alpha/2}\int_{\mathcal{S}_{++}^p} \left(\prod_{i=2}^d |X_{ii}|^{\nu_i}\right) f_{\alpha,(\Sigma^{-1} + 2 \overline{T}_{11})^{-1}}(X) \, \rd X \, \mu_1(\rd T_{11}).
\end{multline*}

Assuming that Conjecture~\ref{conj:1.1} holds true, and applying Lemma~\ref{lem:determinant.power.moments} of~\ref{app:A}, which contains moment formulas for standalone principal minors of Wishart random matrices, one sees that the innermost integral above satisfies
\[
\begin{aligned}
&\int_{\mathcal{S}_{++}^p} \left(\prod_{i=2}^d |X_{ii}|^{\nu_i}\right) f_{\alpha,(\Sigma^{-1} + 2 \overline{T}_{11})^{-1}}(X) \, \rd X \\[-2mm]
&\hspace{30mm}\geq \prod_{i=2}^d \int_{\mathcal{S}_{++}^p} |X_{ii}|^{\nu_i} f_{\alpha,(\Sigma^{-1} + 2 \overline{T}_{11})^{-1}}(X) \, \rd X \\
&\hspace{30mm}= \prod_{i=2}^d 2^{p_i \nu_i} |\{(\Sigma^{-1} + 2 \overline{T}_{11})^{-1}\}_{ii}|^{\nu_i} \frac{\Gamma_{p_i}(\alpha/2 + \nu_i)}{\Gamma_{p_i}(\alpha/2)} \\
&\hspace{30mm}= \prod_{i=2}^d \left[\frac{|\{(\Sigma^{-1} + 2 \overline{T}_{11})^{-1}\}_{ii}|^{\nu_i}}{|\Sigma_{ii}|^{\nu_i}} \times 2^{p_i \nu_i} |\Sigma_{ii}|^{\nu_i} \frac{\Gamma_{p_i}(\alpha/2 + \nu_i)}{\Gamma_{p_i}(\alpha/2)}\right] \\
&\hspace{30mm}\geq \prod_{i=2}^d \left[\inf_{\substack{T_{11}\in \mathcal{S}^{p_1} \\ \Sigma^{-1} + 2 \overline{T}_{11}\in \mathcal{S}_{++}^p}} \frac{|\{(\Sigma^{-1} + 2 \overline{T}_{11})^{-1}\}_{ii}|}{|\Sigma_{ii}|}\right]^{\nu_i} \EE(|\mathfrak{X}_{ii}|^{\nu_i}).
\end{aligned}
\]

By the expression for the Laplace transform of the Wishart distribution in Remark~\ref{rem:Wishart.Lt}, Fubini's theorem, and the Laplace transforms in~\eqref{eq:3.1} for inverse determinant powers, the remaining part of the integral satisfies
\[
\begin{aligned}
\int_{\mathcal{S}_+^{p_1}} |I_p + 2 \overline{T}_{11} \Sigma|^{-\alpha/2} \mu_1(\rd T_{11})
&= \int_{\mathcal{S}_+^{p_1}} \EE\left\{\etr\left(-\overline{T}_{11} \mathfrak{X}\right)\right\} \mu_1(\rd T_{11}) \\
&= \int_{\mathcal{S}_+^{p_1}} \EE\left\{\etr\left(-T_{11} \mathfrak{X}_{11}\right)\right\} \mu_1(\rd T_{11}) \\
&= \EE\left\{\int_{\mathcal{S}_+^{p_1}} \etr\left(-T_{11} \mathfrak{X}_{11}\right) \, \mu_1(\rd T_{11})\right\}
= \EE(|\mathfrak{X}_{11}|^{-\nu_1}).
\end{aligned}
\]
Therefore, one can write
\[
\frac{\EE(|\mathfrak{X}_{11}|^{-\nu_1} \prod_{i=2}^d |\mathfrak{X}_{ii}|^{\nu_i})}{\EE(|\mathfrak{X}_{11}|^{-\nu_1}) \prod_{i=2}^d \EE(|\mathfrak{X}_{ii}|^{\nu_i})} \geq \prod_{i=2}^d \left[\inf_{\substack{T_{11}\in \mathcal{S}^{p_1} \\ \Sigma^{-1} + 2 \overline{T}_{11}\in \mathcal{S}_{++}^p}} \frac{|\{(\Sigma^{-1} + 2 \overline{T}_{11})^{-1}\}_{ii}|}{|\Sigma_{ii}|}\right]^{\nu_i}.
\]

To conclude, it remains to compute the infimum explicitly for each $i\in \{2,\ldots,d\}$. For simplicity of notation, denote the block in position $(i,j)$ inside $\Sigma^{-1}$ by $\Sigma^{ij}$ instead of $(\Sigma^{-1})_{ij}$. From here onwards, fix $i = 2$; then the computation of the infimum follows similarly for the other indices $i\in \{3,\ldots,d\}$. By Lemma~\ref{lem:inverse.block.matrix}~(i) of \ref{app:A}, one has
\[
|\{(\Sigma^{-1} + 2 \overline{T}_{11})^{-1}\}_{22}| = \left|\Sigma^{22} - \begin{bmatrix} \Sigma^{21} & \Sigma^{2,3:d}\end{bmatrix} \begin{bmatrix} \Sigma^{11} + 2 T_{11} & \Sigma^{1,3:d} \\ \Sigma^{3:d,1} & \Sigma^{3:d,3:d}\end{bmatrix}^{-1} \begin{bmatrix} \Sigma^{12} \\ \Sigma^{3:d,2}\end{bmatrix}\right|^{-1}.
\]
The infimum is attained when say $T_{11} = c I_{p_1}$ and $c \to \infty$. By Lemma~\ref{lem:inverse.block.matrix}~(ii) of \ref{app:A} and determinantal properties of Schur complements, one deduces that
\[
\begin{aligned}
\inf_{\substack{T_{11}\in \mathcal{S}^{p_1} \\ \Sigma^{-1} + 2 \overline{T}_{11}\in \mathcal{S}_{++}^p}} |\{(\Sigma^{-1} + 2 \overline{T}_{11})^{-1}\}_{22}|
&= \left|\Sigma^{22} - \begin{bmatrix} \Sigma^{21} & \Sigma^{2,3:d}\end{bmatrix} \begin{bmatrix} 0 & 0 \\ 0 & (\Sigma^{3:d,3:d})^{-1}\end{bmatrix} \begin{bmatrix} \Sigma^{12} \\ \Sigma^{3:d,2}\end{bmatrix}\right|^{-1} \\[-2.5mm]
&= \left|\Sigma^{22} - \Sigma^{2,3:d} (\Sigma^{3:d,3:d})^{-1} \Sigma^{3:d,2}\right|^{-1} \\[1mm]
&= \frac{|\Sigma^{3:d,3:d}|}{|\Sigma^{2:d,2:d}|} = \frac{|\Sigma^{3:d,3:d}| / |\Sigma^{-1}|}{|\Sigma^{2:d,2:d}| / |\Sigma^{-1}|} = \frac{|\Sigma_{1:2,1:2}|}{|\Sigma_{11}|},
\end{aligned}
\]
and thus, for $P_{12} = \Sigma_{11}^{-1/2} \Sigma_{12} \Sigma_{22}^{-1/2}$,
\[
\begin{aligned}
\inf_{\substack{T_{11}\in \mathcal{S}^{p_1} \\ \Sigma^{-1} + 2 \overline{T}_{11}\in \mathcal{S}_{++}^p}} \frac{|\{(\Sigma^{-1} + 2 \overline{T}_{11})^{-1}\}_{22}|}{|\Sigma_{22}|}
&= \frac{|\Sigma_{1:2,1:2}|}{|\Sigma_{11}| |\Sigma_{22}|} = \frac{|\Sigma_{11}| |\Sigma_{22} - \Sigma_{21} \Sigma_{11}^{-1} \Sigma_{12}|}{|\Sigma_{11}| |\Sigma_{22}|} \\[-2mm]
&= \big|I_{p_2} - \Sigma_{22}^{-1/2} \Sigma_{21} \Sigma_{11}^{-1} \Sigma_{12} \Sigma_{22}^{-1/2}\big| \\[1mm]
&= \big|I_{p_2} - P_{12}^{\top} P_{12}\big|.
\end{aligned}
\]
The conclusion follows.
\end{proof}

\begin{proof}[\textbf{Proof of Proposition~\ref{prop:analogue.Theorem.1.3.Zhou.et.al}}]
Fix numbers $\nu_1,\ldots,\nu_d$ such that
\[
(p_i - 1)/2 < \nu_i < \alpha/2 - (p_i - 1)/2
\]
for all $i \in \{1, \ldots, d-1\}$, and $\nu_d\in (0,\infty)$. For simplicity of notation, define
\[
\overline{T} = \left(\bigoplus_{i=1}^{d-1} T_{ii}\right) \midoplus 0_{p_d\times p_d}.
\]
Using the Laplace transforms in~\eqref{eq:3.1} for inverse determinant powers, together with Fubini's theorem, one has
\[
\begin{aligned}
\EE\left\{\left(\prod_{i=1}^{d-1} |\mathfrak{X}_{ii}|^{-\nu_i}\right) |\mathfrak{X}_{dd}|^{\nu_d}\right\}
&= \EE\left[\left\{\int_{\bigtimes_{k=1}^{d-1} \mathcal{S}_+^{p_k}} \prod_{i=1}^{d-1} \etr\left(- T_{ii} \mathfrak{X}_{ii}\right) \mu_i(\rd T_{ii})\right\} \, |X_{dd}|^{\nu_d}\right] \\
&= \int_{\bigtimes_{k=1}^{d-1} \mathcal{S}_+^{p_k}}\int_{\mathcal{S}_{++}^p} \etr(-\overline{T} X) |X_{dd}|^{\nu_d} f_{\alpha,\Sigma}(X) \, \rd X \prod_{i=1}^{d-1} \mu_i(\rd T_{ii}).
\end{aligned}
\]
Now using the expression for the density function of the Wishart distribution in Definition~\ref{def:Wishart}, the right-hand side can be rewritten as
\[
\int_{\bigtimes_{k=1}^{d-1} \mathcal{S}_+^{p_k}} |I_p + 2 \overline{T} \Sigma|^{-\alpha/2} \int_{\mathcal{S}_{++}^p} |X_{dd}|^{\nu_d} f_{\alpha,(\Sigma^{-1} + 2 \overline{T})^{-1}}(X) \, \rd X \prod_{i=1}^{d-1} \mu_i(\rd T_{ii}).
\]
By applying Lemma~\ref{lem:determinant.power.moments} of~\ref{app:A}, which contains moment formulas for standalone principal minors of Wishart random matrices, the innermost integral above satisfies
\begin{equation}\label{eq:4.1}
\int_{\mathcal{S}_{++}^p} |X_{dd}|^{\nu_d} f_{\alpha,(\Sigma^{-1} + 2 \overline{T})^{-1}}(X) \, \rd X
= 2^{p_d \nu_d} \left|\{(\Sigma^{-1} + 2 \overline{T})^{-1}\}_{dd}\right|^{\nu_d} \frac{\Gamma_{p_d}(\alpha/2 + \nu_d)}{\Gamma_{p_d}(\alpha/2)}.
\end{equation}

To bound the last determinant in \eqref{eq:4.1}, use the notation
\[
\Sigma_{\star} = \Sigma_{1:d-1,1:d-1} - \Sigma_{1:d-1,d} \Sigma_{dd}^{-1} \Sigma_{d,1:d-1}
\]
for the Schur complement of $\Sigma_{dd}$ in $\Sigma$. By the formula in Lemma~\ref{lem:inverse.block.matrix}~(ii) of~\ref{app:A} for the inverse of a $2\times 2$ block matrix, one has
\[
\Sigma^{-1} + 2 \overline{T}
=
\begin{bmatrix}
\Sigma_{\star}^{-1} + 2 \bigoplus_{i=1}^{d-1} T_{ii} & - \Sigma_{\star}^{-1} \Sigma_{1:d-1,d} \Sigma_{dd}^{-1} \\[1mm]
-\Sigma_{dd}^{-1} \Sigma_{d,1:d-1} \Sigma_{\star}^{-1} & \Sigma_{dd}^{-1} + \Sigma_{dd}^{-1} \Sigma_{d,1:d-1} \Sigma_{\star}^{-1} \Sigma_{1:d-1,d} \Sigma_{dd}^{-1}
\end{bmatrix}
\equiv
\begin{bmatrix}
A & B \\
C & D
\end{bmatrix}.
\]
Inverting this $2\times 2$ block matrix using the formula in Lemma~\ref{lem:inverse.block.matrix}~(i) of~\ref{app:A}, one finds
\[
\{(\Sigma^{-1} + 2 \overline{T})^{-1}\}_{dd} = (\Sigma_{dd}^{-1} + M)^{-1},
\]
where $M$ is a specific nonnegative definite matrix, namely
\[
\begin{aligned}
M
&= (D - \Sigma_{dd}^{-1}) - C A^{-1} B \\[1mm]
&= \Sigma_{dd}^{-1} \Sigma_{d,1:d-1} \Sigma_{\star}^{-1} \Sigma_{1:d-1,d} \Sigma_{dd}^{-1} \\[-1mm]
&\qquad- \Sigma_{dd}^{-1} \Sigma_{d,1:d-1} \Sigma_{\star}^{-1} \left( \Sigma_{\star}^{-1} + 2 \bigoplus_{i=1}^{d-1} T_{ii} \right)^{-1} \Sigma_{\star}^{-1} \Sigma_{1:d-1,d} \Sigma_{dd}^{-1} \\
&= \Sigma_{dd}^{-1} \Sigma_{d,1:d-1} \left\{ I_p - \Sigma_{\star}^{-1} \left( \Sigma_{\star}^{-1} + 2 \bigoplus_{i=1}^{d-1} T_{ii} \right)^{-1} \right\} \Sigma_{\star}^{-1} \Sigma_{1:d-1,d} \Sigma_{dd}^{-1} \\
&= \Sigma_{dd}^{-1} \Sigma_{d,1:d-1} \left( 2 \bigoplus_{i=1}^{d-1} T_{ii} \right) \left( \Sigma_{\star}^{-1} + 2 \bigoplus_{i=1}^{d-1} T_{ii} \right)^{-1} \Sigma_{\star}^{-1} \Sigma_{1:d-1,d} \Sigma_{dd}^{-1}.
\end{aligned}
\]

Putting the last two equations together, an application of the identity $|A_1 + A_2| \geq |A_1|$, which is valid for any $p_d \times p_d$ nonnegative definite matrices $A_1,A_2\in \mathcal{S}_+^{p_d}$, one gets
\[
\left|\{(\Sigma^{-1} + 2 \overline{T})^{-1}\}_{dd}\right|^{\nu_d} = |\Sigma_{dd}^{-1} + M|^{-\nu_d} \leq |\Sigma_{dd}^{-1}|^{-\nu_d} = |\Sigma_{dd}|^{\nu_d}.
\]
Applying the latter bound into~\eqref{eq:4.1} in conjunction with Lemma~\ref{lem:determinant.power.moments} of~\ref{app:A}, which contains moment formulas for standalone principal minors of Wishart random matrices, one obtains
\[
\int_{\mathcal{S}_{++}^p} |X_{dd}|^{\nu_d} f_{\alpha,(\Sigma^{-1} + 2 \overline{T})^{-1}}(X) \, \rd X \leq 2^{p_d \nu_d} |\Sigma_{dd}|^{\nu_d} \frac{\Gamma_{p_d}(\alpha/2 + \nu_d)}{\Gamma_{p_d}(\alpha/2)} = \EE(|\mathfrak{X}_{dd}|^{\nu_d}).
\]
Hence, the expectation at the beginning of the proof now satisfies
\[
\frac{\EE\{(\, \prod_{i=1}^{d-1} |\mathfrak{X}_{ii}|^{-\nu_i}) |\mathfrak{X}_{dd}|\}}{\EE(|\mathfrak{X}_{dd}|^{\nu_d})} \leq \int_{\bigtimes_{k=1}^{d-1} \mathcal{S}_+^{p_k}} |I_p + 2 \overline{T} \Sigma|^{-\alpha/2} \prod_{i=1}^{d-1} \mu_i(\rd T_{ii}).
\]

By the expression for the Laplace transform of the Wishart distribution in Remark~\ref{rem:Wishart.Lt}, Fubini's theorem, and the Laplace transforms in~\eqref{eq:3.1} for inverse determinant powers, one finds that the last integral satisfies
\[
\begin{aligned}
\int_{\bigtimes_{k=1}^{d-1} \mathcal{S}_+^{p_k}} |I_p + 2 \overline{T} \Sigma|^{-\alpha/2} \prod_{i=1}^{d-1} \mu_i(\rd T_{ii})
&= \int_{\bigtimes_{k=1}^{d-1} \mathcal{S}_+^{p_k}}\int_{\mathcal{S}_{++}^p} \etr\left(-\overline{T} X\right) f_{\alpha,\Sigma}(X) \, \rd X \prod_{i=1}^{d-1} \mu_i(\rd T_{ii}) \\
&= \EE\left\{\int_{\bigtimes_{k=1}^{d-1} \mathcal{S}_+^{p_k}} \prod_{i=1}^{d-1} \etr\left(- T_{ii} \mathfrak{X}_{ii}\right) \mu_i(\rd T_{ii})\right\} \\
&= \EE\left(\prod_{i=1}^{d-1} |\mathfrak{X}_{ii}|^{-\nu_i}\right).
\end{aligned}
\]
The conclusion follows.
\end{proof}

\begin{proof}[\textbf{Proof of Theorem~\ref{thm:generalization.Theorem.2.GO.2023}}]
Throughout this proof, one writes $\overline{\etr}(\cdot) = 1 - \etr(\cdot)$ for short.
The random pairs $(\mathfrak{X}_{11}, \mathfrak{X}_{22})$ and $(\mathfrak{X}_{11}^{\star}, \mathfrak{X}_{22}^{\star})$ satisfy $(\mathfrak{X}_{11}, \mathfrak{X}_{22}) \preceq_{\mathrm{Lt}} (\mathfrak{X}_{11}^{\star}, \mathfrak{X}_{22}^{\star})$ and have identical margins by assumption. Therefore, given any $p \times p$ positive definite matrices $S\in \mathcal{S}_{++}^{p_1}$ and $T\in \mathcal{S}_{++}^{p_2}$, one has
\begin{equation}
\label{eq:4.2}
\begin{aligned}
&\EE \big [ \big\{\overline{\etr}(-S \mathfrak{X}_{11})\big\} \big\{\overline{\etr}(-T \mathfrak{X}_{22})\big\}\big] \\[0.5mm]
&\qquad= 1 - \EE \big\{\etr(-S \mathfrak{X}_{11})\big\} - \EE \big\{\etr(-T \mathfrak{X}_{22})\big\} + \EE \big\{\etr(-S \mathfrak{X}_{11}) \etr(-T \mathfrak{X}_{22})\big\} \\[0.5mm]
&\qquad\geq 1 - \EE \big\{\etr(-S \mathfrak{X}_{11})\big\} - \EE \big\{\etr(-T \mathfrak{X}_{22})\big\} + \EE \big\{\etr(-S \mathfrak{X}_{11}^{\star}) \etr(-T \mathfrak{X}_{22}^{\star})\big\}.
\end{aligned}
\end{equation}
Owing to the fact that the random matrices $\mathfrak{X}_{11}^{\star}$ and $\mathfrak{X}_{22}^{\star}$ are independent, the right-hand term can be rewritten as
\begin{multline*}
1 - \EE \big\{\etr(-S \mathfrak{X}_{11}^{\star})\big\} - \EE \big\{\etr(-T \mathfrak{X}_{22}^{\star})\big\} + \EE \big\{\etr(-S \mathfrak{X}_{11}^{\star})\big\} \EE \big\{\etr(-T \mathfrak{X}_{22}^{\star})\big\} \\[1mm]
= \EE \big[\big\{\overline{\etr}(-S \mathfrak{X}_{11}^{\star})\big\}\big] \EE \big[\big\{\overline{\etr}(-T \mathfrak{X}_{22}^{\star})\big\}\big].
\end{multline*}

Next, observe that if $f$ and $g$ are matrix-variate Bernstein functions with triplets $(A_1,0,\mu_1)$ and $(A_2,0,\mu_2)$, as per Definition~\ref{def:MB}, then one can write $\EE \{ f(\mathfrak{X}_{11}) g(\mathfrak{X}_{22}) \}$ as
\begin{multline*}
\mathrm{tr}(A_1) \mathrm{tr}(A_2)
+ \mathrm{tr}(A_1) \int_{\mathcal{S}_{++}^{p_2}} \EE \big\{\overline{\etr}(-T \mathfrak{X}_{22})\big\} \, \mu_2(\rd T)
+ \mathrm{tr}(A_2) \int_{\mathcal{S}_{++}^{p_1}} \EE \big\{\overline{\etr}(-S \mathfrak{X}_{11})\big\} \, \mu_1(\rd S) \\
+ \int_{\mathcal{S}_{++}^{p_1} \times \mathcal{S}_{++}^{p_2}} \EE \big[\big\{\overline{\etr}(-S \mathfrak{X}_{11})\big\} \big\{\overline{\etr}(-T \mathfrak{X}_{22})\big\}\big] \, \mu_1(\rd S) \, \mu_2(\rd T).
\end{multline*}
Again using the fact that the random pairs $(\mathfrak{X}_{11}, \mathfrak{X}_{22})$ and $(\mathfrak{X}_{11}^{\star}, \mathfrak{X}_{22}^{\star})$ have identical margins by assumption, and in view of inequality~\eqref{eq:4.2}, the above expression is bounded from below by
\begin{multline*}
\mathrm{tr}(A_1) \mathrm{tr}(A_2)
+ \mathrm{tr}(A_1) \int_{\mathcal{S}_{++}^{p_2}} \EE \big\{\overline{\etr}(-T \mathfrak{X}_{22}^{\star})\big\} \, \mu_2(\rd T)
+ \mathrm{tr}(A_2) \int_{\mathcal{S}_{++}^{p_1}} \EE \big\{\overline{\etr}(-S \mathfrak{X}_{11}^{\star})\big\} \, \mu_1(\rd S) \\
+ \int_{\mathcal{S}_{++}^{p_1} \times \mathcal{S}_{++}^{p_2}} \EE \big\{\overline{\etr}(-S \mathfrak{X}_{11}^{\star})\big\} \EE \big\{\overline{\etr}(-T \mathfrak{X}_{22}^{\star})\big\} \, \mu_1(\rd S) \, \mu_2(\rd T),
\end{multline*}
which is the same as $\EE \{f(\mathfrak{X}_{11}^{\star})\} \EE \{g(\mathfrak{X}_{22}^{\star})\}$. Therefore, the claim is proved.
\end{proof}

\begin{proof}[\textbf{Proof of Theorem~\ref{thm:eigenvalues.MTP2}}]
By Theorem~3 of \citet{MR3923912}, the proof of which is based on earlier results of \citet{MR458718}, it is known that the density function for the vector of eigenvalues of a $p\times p$ Wishart random matrix is $\mathrm{MTP}_2$ on $(0,\infty)^p$ according to Definition~\ref{def:MTP2}. Densities in this class have many interesting properties, including the one written in the statement of the theorem, which follows immediately from Eq.~(1.6) of \citet{MR628759}.
\end{proof}

\section{Possible extension: GPI for elliptical laws}\label{sec:elliptical.GPI}

As mentioned in the Introduction, the strong form of the GPI conjecture due to \citet{MR2886380} stipulates that if $\bb{Z} = (Z_1, \ldots, Z_d)$ is a centered Gaussian (column) random vector and if $\alpha_1, \ldots, \alpha_d\in [0,\infty)$, then
\begin{equation}\label{eq:5.1}
\EE \left( \prod_{i=1}^d |Z_i|^{2\alpha_i} \right) \geq \prod_{i=1}^d \EE \big( |Z_i|^{2\alpha_i} \big).
\end{equation}

Suppose, more generally, that the vector $\bb{Z}$ has an elliptical distribution \cite{MR0629795}, so that it can be expressed in the form $\bb{Z} = \sqrt {R} \, A\bb{U}$, where $R$ is a nonnegative random variable, $A$ is a lower-triangular positive definite matrix of size $d \times d$, and $\bb{U}$ is a random vector independent of $R$ and uniformly distributed on the unit sphere in $d$ dimensions, i.e.,
\[
\mathbb{S}^{d-1} = \{ \bb{u}\in \mathbb{R}^d: u_1^2 + \cdots + u_d^2 = 1 \}.
\]

The special case where $R$ is chi-squared with $d$ degrees of freedom corresponds to the multivariate Gaussian distribution with mean zero and correlation matrix $\Sigma = A A^{\top}$. Other choices of distribution for $R$ lead to non-Gaussian elliptical laws, such as the multivariate Student's $t$; see, e.g., \cite{MR1071174} for a book-length treatment of this topic.

Relying exclusively on the stochastic representation $\bb{Z} = \sqrt {R} \, A\bb{U}$ and the independence between $R$ and $\bb{U}$, one can express inequality~\eqref{eq:5.1} in the form
\[
\EE \big( R^{\alpha} \big) \times \EE \left( \prod_{i=1}^d |X_i|^{2\alpha_i} \right) \geq \prod_{i=1}^d \EE (R^{\alpha_i}) \times \prod_{i=1}^d \EE \big( |X_i|^{2\alpha_i} \big),
\]
where $\alpha = \alpha_1 + \cdots + \alpha_d$ and the (column) vector $\bb{X} = (X_1, \ldots, X_d)$ satisfies $\bb{X} = A\bb{U}$. Assuming that $R$ is not identically zero, and regrouping terms differently, one gets
\begin{equation}
\label{eq:5.2}
\EE \left\{ \prod_{i=1}^d \frac{|X_i|^{2\alpha_i}}{ \EE \big( |X_i|^{2\alpha_i} \big)} \right\} \geq \frac{\prod_{i=1}^d \EE (R^{\alpha_i})}{\EE ( R^{\alpha})} .
\end{equation}

Observe that the left-hand term involves only the components of $\bb{X}$ which, for positive integer exponents, are linear combinations of the components of the unit vector $\bb{U}$ uniformly distributed on the sphere $\mathbb{S}^{d-1}$. As for the right-hand ratio, say $Q_R (\alpha_1, \ldots, \alpha_d)$, it depends only on the constants $\alpha_1, \ldots, \alpha_d \in [0, \infty)$ and the distribution of $R$.

Whatever $\alpha_1, \ldots, \alpha_d \in [0, \infty)$, one has $Q_{kR} (\alpha_1, \ldots, \alpha_d) = Q_R (\alpha_1, \ldots, \alpha_d)$ for every possible constant $k \in (0, \infty)$ and $Q_R (\alpha_1, \ldots, \alpha_d) \leq 1$ by the majorization inequality in Section~6 of \citet{MR2266030}; see also \cite{MR4660831} for related work. When $R$ is chi-squared with $d$ degrees of freedom, i.e., $R\sim \mathcal{G}\hspace{0.3mm}(d/2,1/2)$ in the shape-rate parametrization of the Gamma distribution, Eq.~(17.8) in \cite{MR1299979} shows that the ratio on the right-hand side of \eqref{eq:5.2} reduces to
\vspace{-3mm}
\[
Q (\alpha_1, \ldots, \alpha_d) = \frac{\Gamma(d/2)}{2^{\alpha} \Gamma(\alpha + d/2)} \prod_{i=1}^d \frac{2^{\alpha_i} \Gamma(\alpha_i + d/2)}{\Gamma(d/2)}
= \frac{\prod_{i=1}^d \Gamma(\alpha_i + d/2)}{\Gamma(\alpha + d/2) \{\Gamma(d/2)\}^{d-1}}.
\]

Cast in this framework, the strong form of the GPI conjecture amounts to the claim that for any random vector $\bb{U}$ uniformly distributed on $\mathbb{S}^{d-1}$, any lower-triangular positive definite matrix $A$ of size $d \times d$, and any constants $\alpha_1, \ldots, \alpha_d \in [0, \infty)$, one has
\[
\EE \left\{ \prod_{i=1}^d \frac{|X_i|^{2\alpha_i}}{ \EE \big( |X_i|^{2\alpha_i} \big)} \right\} \geq Q (\alpha_1, \ldots, \alpha_d).
\]

Thus if the strong form of the GPI conjecture holds, inequality \eqref{eq:5.1} would also be verified for any other elliptically distributed random vector $\bb{Z}$ for which the density of $R$ is such that $Q_R (\alpha_1, \ldots, \alpha_d) \leq Q (\alpha_1, \ldots, \alpha_d)$ for some $\alpha_1, \ldots, \alpha_d \in [0, \infty)$.

\begin{appendices}

\renewcommand{\thesection}{Appendix~\Alph{section}}

\section{Technical lemmas}\label{app:A}

\renewcommand{\thesection}{\Alph{section}}

The first lemma contains well-known formulas from linear algebra for the inverse of a $2\times 2$ block matrix; see, e.g., Theorem~2.1 of \citet{MR1873248}.

\begin{lemma}\label{lem:inverse.block.matrix}
Let $M$ be a matrix of width at least $2$ which is partitioned into a $2\times 2$ block matrix as follows:
\[
M =
\begin{bmatrix}
M_{11} & M_{12} \\
M_{21} & M_{22}
\end{bmatrix}.
\]
\begin{itemize}\setlength\itemsep{0em}
\item[(i)]
If $M_{11}$ is invertible and the Schur complement $M_{\star} = M_{22} - M_{21} M_{11}^{-1} M_{12}$ is invertible, then
\vspace{-1mm}
\[
M^{-1} =
\begin{bmatrix}
M_{11}^{-1} + M_{11}^{-1} M_{12} M_{\star}^{-1} M_{21} M_{11}^{-1} & - M_{11}^{-1} M_{12} M_{\star}^{-1} \\[1mm]
- M_{\star}^{-1} M_{21} M_{11}^{-1} & M_{\star}^{-1}
\end{bmatrix}.
\]
\item[(ii)]
If $M_{22}$ is invertible and the Schur complement $M_{\star} = M_{11} - M_{12} M_{22}^{-1} M_{21}$ is invertible, then
\vspace{-1mm}
\[
M^{-1} =
\begin{bmatrix}
M_{\star}^{-1} & - M_{\star}^{-1} M_{12} M_{22}^{-1} \\[1mm]
- M_{22}^{-1} M_{21} M_{\star}^{-1} & M_{22}^{-1} + M_{22}^{-1} M_{21} M_{\star}^{-1} M_{12} M_{22}^{-1}
\end{bmatrix}.
\]
\end{itemize}
\end{lemma}

The second lemma contains moment formulas for standalone principal minors of Wishart random matrices. It is a consequence, e.g., of Corollary~3.2.6 and p.\,101 of \citet{MR652932} but has been known since the work of \citet{doi:10.2307/2331979} when the degree-of-freedom parameter, $\alpha$, is integer-valued. Refer to \citet{MR2458187} for moment formulas involving off-diagonal minors of Wishart random matrices.

\begin{lemma}\label{lem:determinant.power.moments}
For any positive integers $d$ and $p_1, \ldots, p_d \in \N$, let $p = p_1 + \cdots + p_d$ and fix any $\alpha \in (p-1,\infty)$. Let also $\Sigma\in \mathcal{S}_{++}^p$ be any $p \times p$ positive definite matrix and
\[
\mathfrak{X}
= \begin{bmatrix}
\mathfrak{X}_{11} & \cdots & \mathfrak{X}_{1d} \\
\vdots & \ddots & \vdots \\
\mathfrak{X}_{d1} & \cdots & \mathfrak{X}_{dd} \\
\end{bmatrix} \sim\mathcal{W}_p(\alpha,\Sigma), \quad \text{with}~~ \Sigma =
\begin{bmatrix}
\Sigma_{11} & \cdots & \Sigma_{1d} \\
\vdots & \ddots & \vdots \\
\Sigma_{d1} & \cdots & \Sigma_{dd} \\
\end{bmatrix},
\]
as per Definition~\ref{def:Wishart}. Then, for every $i \in \{1, \ldots, d\}$ and $\nu_i > -\alpha/2 + (p_i - 1)/2$, one has
\[
\EE(|\mathfrak{X}_{ii}|^{\nu_i}) = 2^{p_i \nu_i} |\Sigma_{ii}|^{\nu_i} \frac{\Gamma_{p_i}(\alpha/2 + \nu_i)}{\Gamma_{p_i}(\alpha/2)}.
\]
\end{lemma}

\renewcommand{\thesection}{Appendix \Alph{section}}

\section{Conditions for the finiteness of certain expectations}\label{app:B}

\renewcommand{\thesection}{\Alph{section}}

The goal of this appendix is to determine necessary and sufficient conditions on the reals $\nu_1,\ldots,\nu_d\in (0,\infty)$ under which the expectations in the statement of Corollary~\ref{cor:GPI.negative.powers.determinants.Wishart} are finite. These conditions can also be applied in Proposition~\ref{prop:analogue.Theorem.1.3.Zhou.et.al}.

\begin{lemma}\label{lem:finite.exp}
In the setting of Corollary~\ref{cor:GPI.negative.powers.determinants.Wishart}, the expectation
\[
\EE \left(\prod_{i=1}^d |\mathfrak{X}_{ii}|^{-\nu_i}\right)
\]
is infinite if there exists at least one index $i\in \{1,\ldots,d\}$ for which $\nu_i \geq \alpha/2 - (p_i - 1)/2$. It is finite if $(p_i - 1)/2 < \nu_i < \alpha/2 - (p_i - 1)/2$ for every integer $i \in \{1, \dots, d\}$.
\end{lemma}

\begin{proof}[\textbf{Proof of Lemma~\ref{lem:finite.exp}}]
As explained in the paragraphs leading up to Corollary~\ref{cor:GPI.negative.powers.determinants.Wishart}, if $\nu_i > (p_i - 1)/2$ for all $i\in \{1,\ldots,d\}$, then by iterating the lower bound in Theorem~\ref{thm:generalization.Theorem.3.2.Wei}, one obtains
\[
\prod_{i=1}^d \EE\big(|\mathfrak{X}_{ii}|^{-\nu_i}\big) \leq \EE\left(\prod_{i=1}^d |\mathfrak{X}_{ii}|^{-\nu_i}\right).
\]
By Lemma~\ref{lem:determinant.power.moments} of \ref{app:A}, the left-hand side is infinite if there exists at least one index $i\in \{1,\ldots,d\}$ for which $\nu_i \geq \alpha/2 - (p_i - 1)/2$. This proves the first claim.

For the second claim, using the proof of Theorem~\ref{thm:generalization.Theorem.3.2.Wei} as a starting point, one has
\begin{align}
\label{eq:B.1}
\EE\left(\prod_{i=1}^d |\mathfrak{X}_{ii}|^{-\nu_i}\right)
= \int_{\mathcal{S}_+^{p_1} \times \cdots \times \mathcal{S}_+^{p_d}} \left|I_p + 2 \left(\bigoplus_{i=1}^d T_{ii}\right) \Sigma\right|^{-\alpha/2} \prod_{i=1}^d \frac{|T_{ii}|^{\nu_i - (p_i + 1)/2}}{\Gamma_{p_i}(\nu_i)} \, \rd T_{ii}.
\end{align}
By applying the nonsingular transformation,
\[
T_{dd} \to \Sigma_{dd}^{-1/2} T_{dd} \Sigma_{dd}^{-1/2},
\]
the Jacobian and determinant factor $|T_{dd}|^{\nu_d - (p_d + 1)/2}$ give rise to powers of $|\Sigma_{dd}|$, which one can remove from the integrand. This reduces the problem of determining conditions for the finiteness of the integral in \eqref{eq:B.1} to the special case in which $\Sigma_{dd} = I_{p_d}$.

For further reference, note that because $\Sigma_{dd} = I_{p_d}$, the Schur complement
\[
\Sigma_{\star} = \Sigma_{1:d-1,1:d-1} - \Sigma_{1:d-1,d} \Sigma_{d,1:d-1}
\]
is positive definite. Now write
\[
\begin{aligned}
\left|I_p + \left(\bigoplus_{i=1}^d T_{ii}\right)\Sigma\right|
&= \left|I_p + \left(\bigoplus_{i=1}^d T_{ii}\right)^{1/2} \Sigma \left(\bigoplus_{i=1}^d T_{ii}\right)^{1/2}\right| \\
&=
\begin{vmatrix}
\begin{bmatrix}
I_{p-p_d} & 0 \\ 0 & I_{p_d}
\end{bmatrix}
\!+\!
\begin{bmatrix}
\big(\bigoplus_{i=1}^{d-1} T_{ii}\big)^{1/2} & 0 \\ 0 & T_{dd}^{1/2}
\end{bmatrix}
\Sigma
\begin{bmatrix}
\big(\bigoplus_{i=1}^{d-1} T_{ii}\big)^{1/2} & 0 \\ 0 & T_{dd}^{1/2}
\end{bmatrix}
\end{vmatrix}.
\end{aligned}
\]
On inserting
\[
\Sigma =
\begin{bmatrix}
\Sigma_{1:d-1,1:d-1} & \Sigma_{1:d-1,d} \\
\Sigma_{d,1:d-1} & I_{p_d}
\end{bmatrix}
\]
into the latter determinant and simplifying the expression, one obtains
\[
\left|I_p + \left(\bigoplus_{i=1}^d T_{ii}\right)\Sigma\right|
=
\begin{vmatrix}
I_{p-p_d} + A & B T_{dd}^{1/2} \\
T_{dd}^{1/2} B^{\top} & I_{p_d} + T_{dd}
\end{vmatrix},
\]
where
\[
A = \left(\bigoplus_{i=1}^{d-1} T_{ii}\right)^{1/2} \Sigma_{1:d-1,1:d-1} \left(\bigoplus_{i=1}^{d-1} T_{ii}\right)^{1/2} \quad \text{and} \quad
B = \left(\bigoplus_{i=1}^{d-1} T_{ii}\right)^{1/2} \Sigma_{1:d-1,d}.
\]

By applying Schur complements, one finds that
\[
\begin{aligned}
\left|I_p + \left(\bigoplus_{i=1}^d T_{ii}\right)\Sigma\right|
= |I_{p_d} + T_{dd}| \times |I_{p-p_d} + A - B T_{dd}^{1/2} (I_{p_d} + T_{dd})^{-1} T_{dd}^{1/2} B^{\top}|.
\end{aligned}
\]
In the latter of these two determinants, there appears the term $T_{dd}^{1/2} (I_{p_d} + T_{dd})^{-1} T_{dd}^{1/2}$. It is straightforward to prove that, for any positive definite matrix $T_{dd}\in \mathcal{S}_{++}^{p_d}$,
\[
T_{dd}^{1/2} (I_{p_d} + T_{dd})^{-1} T_{dd}^{1/2} < I_{p_d}
\]
in the positive definite ordering. Therefore
\[
\begin{aligned}
|I_{p-p_d} + A - B T_{dd}^{1/2} (I_{p_d} + T_{dd})^{-1} T_{dd}^{1/2} B^{\top}|
&\geq |I_{p-p_d} + A - B B^{\top}| \\
&= \left|I_{p-p_d} + \left(\bigoplus_{i=1}^{d-1} T_{ii}\right)^{1/2} \Sigma_{\star} \left(\bigoplus_{i=1}^{d-1} T_{ii}\right)^{1/2}\right| \\
&= \left|I_{p-p_d} + \left(\bigoplus_{i=1}^{d-1} T_{ii}\right) \Sigma_{\star}\right|.
\end{aligned}
\]
In turn, one has
\[
\left|I_p + \left(\bigoplus_{i=1}^d T_{ii}\right)\Sigma\right| \geq |I_{p_d} + T_{dd}| \times \left|I_{p-p_d} + \left(\bigoplus_{i=1}^{d-1} T_{ii}\right) \Sigma_{\star}\right|.
\]

By applying this inequality in \eqref{eq:B.1}, one obtains
\[
\begin{aligned}
\EE\left(\prod_{i=1}^d |\mathfrak{X}_{ii}|^{-\nu_i}\right)
&\leq \int_{\mathcal{S}_+^{p_d}} |I_{p_d} + T_{dd}|^{-\alpha/2} \frac{|T_{dd}|^{\nu_d - (p_d + 1)/2}}{\Gamma_{p_d}(\nu_d)} \rd T_{dd} \\
&\quad\times \int_{\mathcal{S}_+^{p_1} \times \cdots \times \mathcal{S}_+^{p_{d-1}}}
\left|
I_{p-p_d} + \left(\operatornamewithlimits{\bigoplus}\limits_{i=1}^{d-1} T_{ii}\right) \Sigma_{\star}
\right|^{-\alpha/2} \prod_{i=1}^{d-1} \frac{|T_{ii}|^{\nu_i - (p_i + 1)/2}}{\Gamma_{p_i}(\nu_i)} \, \rd T_{ii}.
\end{aligned}
\]
The first integral converges for all $(\alpha,p_d,\nu_d)$ satisfying $(p_d - 1)/2 < \nu_d < \alpha/2 - (p_d - 1)/2$ according to Definition~5.2.2 of \citep{Gupta_Nagar_1999} for the matrix-variate Beta-type-II distribution, or equivalently, to Theorem~8.2.8 of \citet{MR652932} for the Fisher ($F$) distribution. By repeating the above argument successively for $i = d-1$, then $i = d-2$, etc., all the way down to $i = 1$, one finds that the expectation is finite if $(\alpha,p_i,\nu_i)$ satisfies
\[
(p_i - 1)/2 < \nu_i < \alpha/2 - (p_i - 1)/2
\]
for every integer $i\in \{1,\ldots,d\}$. This proves the second claim.
\end{proof}

\renewcommand{\thesection}{Appendix \Alph{section}}

\section{Evaluation of the integrals in Corollary~\ref{cor:GPI.negative.powers.determinants.Wishart}}\label{app:C}

\renewcommand{\thesection}{\Alph{section}}

The goal of this appendix is to study the integrals
\[
\mathcal{I}_{p_i}(M_{ii}) = \int_{\mathcal{S}_{++}^{p_i}} \frac{|T_{ii}|^{\nu_i - (p_i + 1)/2}}{\big|I_{p_i} + \sqrt{2} \, T_{ii}^{1/2} M_{ii}\big|^{\alpha}} \, \rd T_{ii},
\]
which appear, for every $i\in \{1,\ldots,d\}$ and $M_{ii}\in \mathcal{S}_{++}^{p_i}$, in the upper bound in Corollary~\ref{cor:GPI.negative.powers.determinants.Wishart}.

It will be shown below that the Jacobian determinant of the transformation $T_{ii} \mapsto X^2$ with respect to $X$ is
\begin{equation}\label{eq:Jacobian.x.2}
J_{p_i}(X) = 2^{p_i} |X| \prod_{1 \leq i < j \leq p_i} (\lambda_i + \lambda_j).
\end{equation}
Assuming for now that this is true, the quantity $\prod_{1 \leq i < j \leq p_i} (\lambda_i + \lambda_j)$ is a linear combination of monomial symmetric functions, and thus a linear combination of zonal polynomials. Specifically, there exist some constants $a_{\bb{\kappa}}\in \R$ such that
\[
J_{p_i}(X) = 2^{p_i} |X| \sum_{\bb{\kappa}} a_{\bb{\kappa}} C_{\bb{\kappa}}(X),
\]
where, from here onwards, the sum $\sum_{\bb{\kappa}}$ is over all partitions $\bb{\kappa} = (k_1,\ldots,k_{p_i})$ satisfying $k_1 \geq \dots \geq k_{p_i} \geq 0$ and $k_1 + \dots + k_{p_i} = p_i(p_i+1)/2$. Hence, one finds
\[
\mathcal{I}_{p_i}(M_{ii}) = 2^{p_i} \sum_{\bb{\kappa}} a_{\bb{\kappa}} \int_{\mathcal{S}_{++}^{p_i}} \frac{|X|^{2\nu_i - p_i}}{\big|I_{p_i} + \sqrt{2} \, X M_{ii}\big|^{\alpha}} \, C_{\bb{\kappa}}(X) \, \rd X.
\]

Applying the transformation $X \mapsto (2 M_{ii}^2)^{-1/4} X (2 M_{ii}^2)^{-1/4}$, which has Jacobian determinant $|2 M_{ii}^2|^{-(p_i + 1)/4}$ by Theorem~2.1.6 of~\citet{MR652932}, it follows that
\[
\mathcal{I}_{p_i}(M_{ii}) = 2^{p_i} |2 M_{ii}^2|^{-\nu_i + (p_i - 1)/4} \sum_{\bb{\kappa}} a_{\bb{\kappa}} \int_{\mathcal{S}_{++}^{p_i}} \frac{|X|^{2\nu_i - p_i}}{\big|I_{p_i} + X\big|^{\alpha}} \, C_{\bb{\kappa}}(M_{ii}^{-1} X / \sqrt{2}) \, \rd X.
\]
From Eq.~(13) of~\citet{MR190965}, or equivalently from Eq.~(1.5.18) of \citet{Gupta_Nagar_1999}, the summation above is equal to
\[
\sum_{\bb{\kappa}} a_{\bb{\kappa}} \, \frac{\Gamma_{p_i}\{2\nu_i - (p_i - 1)/2, \bb{\kappa}\} \Gamma_{p_i}[\alpha - \{2\nu_i - (p_i - 1)/2\}, -\bb{\kappa}]}{\Gamma_{p_i}(\alpha)} \, C_{\bb{\kappa}}(M_{ii}^{-1} / \sqrt{2}),
\]
if and only if $(p_i - 1)/2 < \nu_i < \alpha/2 - p_i (p_i + 1)/4$, where, for every integer $m\in \N$ and reals $a > (m - 1)/2$, $b > (m - 1)/2 + k_1$,
\[
\begin{aligned}
\Gamma_m(a, \bb{\kappa})
&= \pi^{m(m-1)/4} \prod_{j=1}^m \Gamma\{a + k_j - (j - 1)/2\}, \\
\Gamma_m(b, -\bb{\kappa})
&= \pi^{m(m-1)/4} \prod_{j=1}^m \Gamma\{a - k_j - (m - j)/2\}.
\end{aligned}
\]
This proves the claim made in Remark~\ref{rem:finite.exp}.

\newpage
To complete this section, it remains to prove Eq.~\eqref{eq:Jacobian.x.2}.

\begin{proof}[Proof of Eq.~\eqref{eq:Jacobian.x.2}]
For simplicity of notation, let $p_i = p\in \N$. The following approach is motivated by the derivation of a Jacobian given by Herz \cite[Lemma 3.7, p.~495]{MR69960}.  The idea is to calculate the Laplace transform of the measure arising from transforming the Lebesgue measure using the transformation $X \mapsto X^2$.  Thus, for $T\in \mathcal{S}_{++}^p$, consider the integral
\[
\mathcal{L}_p(T) = \int_{\mathcal{S}^p_{++}} \etr(-TX^{1/2}) \rd X.
\]
Now introduce the transformation $X = H^{\top}\Lambda H$, where $\Lambda = \textrm{diag}(\lambda_1,\ldots,\lambda_p)$ is a diagonal matrix containing the eigenvalues of $X$, and $H \in O(p)$, the orthogonal group.  It is well-known (Muirhead, Theorem 3.2.17, p.~104) that the Jacobian of the transformation $X \mapsto (\Lambda,H)$ is
\[
V(\Lambda) = c \prod_{1 \leq i < j \leq p} |\lambda_i - \lambda_j|,
\]
where $c$ is a constant whose explicit value is not needed. Since $(H^{\top}\Lambda H)^{1/2} = H^{\top}\Lambda^{1/2}H$, then
\[
\mathcal{L}_p(T) = \int_{O(p)} \int_{\Lambda \in \mathcal{S}^p_{++}} \etr(-TH^{\top}\Lambda^{1/2}H) V(\Lambda) \rd \Lambda \rd H.
\]
Now make the transformation $\Lambda \mapsto \Lambda^2$, i.e., $\lambda_i \mapsto \lambda_i^2$ for each $i \in \{ 1 ,\dots, p \}$; the Jacobian of this transformation is
$
\prod_{i=1}^p 2\lambda_i.
$
Therefore
\[
\mathcal{L}_p(T) = \int_{O(p)} \int_{\Lambda \in \mathcal{S}^p_{++}} \etr(-TH^{\top}\Lambda H) V(\Lambda^2) \Big(\prod_{i=1}^p 2\lambda_i\Big) \rd \Lambda \rd H.
\]
Considering that
\[
V(\Lambda^2) = c \prod_{1 \leq i < j \leq p} |\lambda_i^2 - \lambda_j^2| = c \prod_{1 \leq i < j \leq p} |(\lambda_i - \lambda_j)(\lambda_i + \lambda_j)| = V(\Lambda) \prod_{1 \leq i < j \leq p} (\lambda_i + \lambda_j),
\]
one has
\[
\mathcal{L}_p(T) = \int_{O(p)} \int_{\Lambda \in \mathcal{S}^p_{++}} \etr(-TH^{\top}\Lambda H) V(\Lambda) \Big(\prod_{i=1}^p 2\lambda_i\Big) \bigg\{\prod_{1 \leq i < j \leq p} (\lambda_i + \lambda_j)\bigg\} \rd \Lambda \rd H.
\]
Observe that the expression
\[
\Big(\prod_{i=1}^p 2\lambda_i\Big) \bigg\{\prod_{1 \leq i < j \leq p} (\lambda_i + \lambda_j)\bigg\},
\]
is a \textit{symmetric} function of the eigenvalues, and therefore it is an orthogonally invariant function of $X$.  Denoting this function by $J_p(X)$, one obtains $J_p(X) = J_p(H^{\top}XH)$, and therefore
\[
\mathcal{L}_p(T) = \int_{O(p)} \int_{\Lambda \in \mathcal{S}^p_{++}} \etr(-TH^{\top}\Lambda H) J_p(\Lambda) V(\Lambda) \rd \Lambda \rd H.
\]
Now making the reverse transformation $(\Lambda,H) \mapsto X$, one obtains
\[
\mathcal{L}_p(T) = \int_{\mathcal{S}^p_{++}} \etr(-TX) J_p(X) \rd X.
\]
This proves that the Jacobian of the transformation $X \mapsto X^2$ is $J_p(X)$.  Note also that
\[
J_p(X) = 2^p |X| \prod_{1 \leq i < j \leq p} (\lambda_i + \lambda_j).
\]
This concludes the proof.
\end{proof}

\end{appendices}

\section*{Acknowledgments}
\addcontentsline{toc}{section}{Acknowledgments}
\noindent
The authors thank the reviewers for their careful reading and positive feedback.

\section*{Funding}
\addcontentsline{toc}{section}{Funding}
\noindent
Genest's research is funded in part by the Canada Research Chairs Program (Grant no.~950--231937) and the Natural Sciences and Engineering Research Council of Canada (RGPIN-2024-04088). Ouimet's funding was made possible through a contribution to Genest's research program from the Trottier Institute for Science and Public Policy.

\addcontentsline{toc}{section}{References}

\bibliographystyle{plainnat}
\bibliography{GOR_2024_matrix_variate_GPI_bib}

\end{document}